\documentclass{birkjour}
\usepackage{amsfonts,amssymb,amsmath,mathtools,hyperref}
\usepackage[english]{babel}
\usepackage{bigints}
\usepackage{amssymb}
\usepackage{pdfsync}

\usepackage[normalem]{ulem}
\usepackage{cancel}

\usepackage{fixltx2e}
\usepackage[all,error]{onlyamsmath}

\usepackage{amssymb,amsfonts,amsmath,
amsrefs,enumerate,color,url,hyperref,
mathbbol,dsfont}
\usepackage{mathrsfs}
\usepackage{tikz}

\usepackage[T1]{fontenc}
\usepackage[cp1250]{inputenc}

\usepackage[strict=true]{csquotes}

\numberwithin{equation}{section}

\usepackage{relsize,bigints}

\usepackage{graphicx,tikz}

\usepackage{color}

\newtheorem{thm}{Theorem}[section]
\newtheorem{lemma}[thm]{Lemma}
\newtheorem{proposition}[thm]{Proposition}
\newtheorem{corollary}[thm]{Corollary}

\newtheorem{remark}[thm]{Remark}

\newcommand{\eps}{\varepsilon}

\DeclareMathAlphabet{\mathpzc}{OT1}{pzc}{m}{it}

\newcommand{\ka}{{\mathpzc k}}
\newcommand{\kad}{{\mathpzc K}}

\newcommand{\mquad}[1]{\qquad \text{#1} \qquad}

\newcommand{\seq}[1]{\left ( #1\right )_{i\in \kad}}

\newcommand{\e}{\mathrm e}
\newcommand{\ud }{\, \mathrm d}
\newcommand{\mud}{\mathrm d}
\newcommand{\sem}[1]{\left ( \e^{t#1} \right)_{t \ge 0}}
\newcommand{\dom}[1]{\mathcal D(#1)}
\newcommand{\grae}{\lim_{\eps \to 0}}
\newcommand{\lam}{\lambda}
\newcommand{\rla}{R_\lam} 
 
\newcommand{\rlae}{R_{\lam,\eps}}
\newcommand{\rlaz}{R_{\lam,0}}
 
\newcommand{\gra}{\lim_{n\to \infty}}

\newcommand{\rez}[1]{\left (\lam - #1\right )^{-1}}

\newcommand{\mc}{\mathcal}

\newcommand{\sui}{\sum_{i\in \kad}}
\newcommand{\R}{\mathbb R}

\newcommand{\E}{\mathrm{E}}

\newcommand{\slam}{\sqrt{2\lam}}

\usepackage{graphicx,tikz}
\usetikzlibrary{arrows}
\definecolor{ngreen}{HTML}{006400}

\title[]{On Portenko's approximation of  skew \\ Brownian motion}

\author{Adam Bobrowski} 
\address{Lublin University of Technology\\Department of Mathematics\\Nadbystrzycka 38A\\20-618 Lublin, Poland}
\author{Andrey Pilipenko}
 \address{University of Geneva, 
Section de math\'ematiques, 
UNI DUFOUR\\
24, rue du G\`en\`eral Dufour\\
Case postale 64, 1211 Geneva 4, Switzerland\\and \\
Institute of Mathematics of the Ukrainian National Academy of Sciences\\
 Tereschenkivska st. 3, Kiev-4, 01601, Ukraine}

\begin{document}

\begin{abstract} From the perspective of the theory of operator semigroups, we reflect back on the classical theorem of Portenko devoted to approximation of skew Brownian motion.  
The theorem  says that by concentrating the power of drift of a diffusion process around a point one obtains an equivalent of a semi-permeable membrane at this point, described by skew Brownian motion's boundary condition. We prove convergence of the corresponding Feller semigroups and in doing so, generalize Portenko's theorem to the case of  the Walsh processes on star graphs. Our analysis leads through singular perturbations of Sturm--Liouville equations, and 
reveals that as a result of Portenko-type approximation parameters of Walsh processes are transformed in a simple and  elegant manner.
 \end{abstract}



\maketitle
\vspace{-1.5cm}
\section{Introduction}
Skew Brownian motion has been introduced in It\^o and McKean's monograph \cite{ito} (see Problem 1, p. 115 there) as a modification of the standard Brownian motion in which the sign of each excursion from $0$ is chosen (independently from other excursions) to be positive with probability $p\in (0,1)$, and negative with probability $1-p$; this characteristic property of skew Brownian motion is expressed in the boundary condition (see further on for details)
\[ pf'(0+) = (1-p) f'(0-).\]
Later the process, together with its generalizations, including the Walsh's spider process, gained its due recognition with the seminal paper of J.B. Walsh \cite{walsh}. It is still a plausible object of study, perhaps because of surprising symmetries and regularities in its structure \cites{abtk,aberman,fromsotows,abernachr,kinetic,barlow,adasma,bookandrey,kostrykin2012,kostrykin2,lejayskew,manyor, yor97}.

In the somewhat forgotten work  \cite{portenko} of the second part of 1970s, described also in \cite{portenkokniga}*{Thm. 3.4., p. 146} and preceding  the paper of J.B. Walsh,  N.I.~Portenko also, independently,  defined what turns out to be the skew Brownion motion, as a particular case of his generalized diffusions, and found the following approximation.

Let   $a_\eps, \eps >0$ be  nonnegative, continuous and integrable functions on $\R$ with the following property: we have
$\sup_{\eps>0}\int_\R a_\eps (z) \ud z <\infty$, and 
 there is a constant $\alpha\ge 0$ such that 
\begin{equation}\label{intro:1} \grae \int_x^y a_\eps (z) \ud z  = \alpha [xy<0], \qquad \, x,y\in \R\setminus \{0\}, x<y, \end{equation} 
 where $[\cdot]$ is the Iverson bracket. Such a family of functions can be constructed, for example, as follows: given a nonnegative, continuous and integrable function $a$ on $\R$ we define \begin{equation}\label{zero}a_\eps (x) = \eps^{-1} a(\eps^{-1} x), \qquad x \in \R;\end{equation} then \eqref{intro:1} holds with $\alpha \coloneqq \int_{\R} a(x)\ud x$. Portenko has proved that the Feller semigroups generated by 
 \[ \mathfrak A_\eps f = \tfrac 12 f'' + a_\eps f' \]
 where $a_\eps $s  have property \eqref{intro:1}, converge, as $\eps \to 0$, to the skew Brownian motion with parameter 
 \begin{equation} \widetilde p=\tfrac {\e^{\alpha}}{\e^{\alpha} + \e^{-\alpha}}.\label{intro:2} \end{equation}  
 Since $\widetilde p\ge \frac 12$, this  skew Brownian motion can be interpreted as having an infinitesimal drift to the right at $x=0$, that is, a tendency to  assign positive rather than negative signs of excursions from $0$.

It is the goal of this note to give a semigroup-theoretic proof of the Portenko theorem and in doing so to generalize it. 
To this end, in Section \ref{sec:tmt} we consider diffusion processes on a finite interval with fixed diffusion coefficient, reflecting barriers at either end, a semi-permeable membrane in the interval center, and a drift that concentrates more and more around this center; we show that the limit process is a Brownian motion with transformed permeability coefficient of the membrane (see Theorem \ref{thm:glo}). The proof, contained in Section \ref{sec:pot}, leads through the analysis of the resolvents of the  generators of the approximating processes  and has the advantage of showing that the Portenko theorem involves an elegant singular perturbation of the related Sturm--Liouville equations.

In Section \ref{sec:gene} the result is generalized further (see Theorem \ref{thm:glo2}) to the case where the underlying processes have values in a finite star graph. Interestingly, the proof of Theorem \ref{thm:glo2} is not a straightforward adaptation of the proof of Theorem \ref{thm:glo} to the more complex situation. Rather, it involves deviceing a different, more handy, formula for the resolvents of approximating processes and that of the limit process (see formulae \eqref{pta:7}--\eqref{pta:8}, \eqref{pta:11} and Remark \ref{rem:2}).

 In Section \ref{sec:infinite} we show that the results obtained on finite graphs can be extended to the case of infinite graphs as well. Here, we provide two proofs. In the first (see Section \ref{sec:coxe}), we treat 
Theorem \ref{thm:glo2} as a key lemma, and examine the distributions of the processes involved.  In the second (see Section \ref{sec:corbis}), in keeping with the spirit of the paper, we come back to the analysis of resolvents.

Before completing this introduction, we mention the paper by Le Gall \cite{legall}, which in a sense provides a converse to Portenko's approximation. To elaborate on this succinct statement: Portenko's theorem reveals that by concentrating the power of drift around a point one obtains an equivalent of a semi-permeable membrane at this point, described by skew Brownian motion's boundary condition. Le Gall goes in the other direction: he considers processes that need to pass through a sequence of semi-permeable membranes, and shows that (under natural conditions) in the limit, as the membranes get closer to each other and their number grows appropriately, the processes converge to a one with `regular' drift. See also \cites{makno,andreynew}. 
 
\newcommand{\askew}{\mathfrak A_{\textnormal {skew}}} 
\newcommand{\awalsh}{\mathfrak A_{\textnormal {Walsh}}} 
\section{Portenko-type approximation of skew Brownian motion}\label{sec:tmt}

 Let $C[-r,r]$ be the space of continuous functions on the closed interval $[-r,r]$, where $r>0$ is fixed throughout Sections \ref{sec:tmt}--\ref{sec:gene}. Let also $p\in (0,1)$ be given, and let $\mc D_{p,\eps}\subset C[-r,r], \eps >0$ be the set of $f$ such that 
\begin{itemize} 
\item [(a) ] $f$ is twice continuously differentiable at either of the subintervals $[-r,0]$ and $[0,r]$ separately  (with one-sided derivatives at each intervals' ends),
\item [(b) ] $f'(-r)=f'(r)=0$ and the one-sided derivatives at $0$ are related by $pf'(0+)= (1-p)f'(0-)$,
\item [(c) ] we have $\frac 12 f''(0+) + a_\eps (0) f'(0+) = \frac 12 f''(0-) + a_\eps (0) f'(0-)$, 
\end{itemize}
where $a_\eps, \eps >0$ is a family of functions described in the introduction. For $f\in \mc D_{p,\eps}$ it makes sense to define 
\[ \mathfrak A_\eps f = \tfrac 12 f'' + a_\eps f' \]
and it turns out that the so-defined $\mathfrak A_\eps$s are Feller generators   $C[-r,r]$ (see Section \ref{sec:cor} for details). 

Our first main theorem speaks of the convergence of the semigroups generated by $\mathfrak A_\eps$s. 

\begin{thm}\label{thm:glo} Let $\sem{\mathfrak A_\eps}$ be the semigroup generated by $\mathfrak A_\eps$. Then  
\[ \grae \e^{t\mathfrak A_\eps} = \e^{t\askew}, \]
strongly and uniformly with respect to $t$ in compact subsets of $[0,\infty)$, where $\sem{\askew}$ is the semigroup related to the skew Brownian motion with parameter 
\begin{equation}\label{mthm:1} \widetilde p \coloneqq \tfrac p{p+(1-p)\e^{-2\alpha}} \end{equation} 
(which is no smaller than $p$) and reflecting barriers at $-r$ and $r$. \end{thm}

The semigroup $\sem{\askew}$ featured in the theorem is generated by the operator $\askew$ defined as follows. Its domain is composed of $f\in C[-r,r]$ such that  
\begin{itemize} 
\item [(a) ] $f$ is twice continuously differentiable at either of the subintervals $[-r,0]$ and $[0,r]$ separately  (with one-sided derivatives at each intervals' ends),
\item [(b) ] $f'(-r)=f'(r)=0$ and the one-sided derivatives at $0$ are related by $\widetilde pf'(0+)= (1-\widetilde p)f'(0-)$,
\item [(c) ] we have $f''(0+) = f''(0-)$, 
\end{itemize}
and for such $f$ we agree that $\askew f = \frac 12 f''$. The fact that $\askew$ is a Feller generator can be seen as a by-product of the generation theorem for $\mathfrak A_\eps, \eps >0$ --- it suffices to consider $a_\eps =0$ and take $\widetilde p$ instead of $p$.  

In Remark \ref{rem:2} we argue that the Feller process generated by $\askew$ is a limit, as $\delta \to 0+$, of Brownian motions in  $[-r,r]$ that are reflected  at the interval's ends and upon any visit at $0$ jump to $\delta$ or $-\delta$ with probabilities $\widetilde p$ and $1-\widetilde p$, respectively. This supports our view of the process as a skew Brownian motion on $[-r,r]$, reflected at the interval's ends. As a further, and stronger, support of this fact, in Section \ref{rem:andrzeja} we construct the process by means of excursions of the reflected Brownian motion on $[0,r]$.

The Trotter--Kato--Neveu theorem \cites{kniga,knigazcup,ethier,goldstein,pazy,kallenbergnew} is our main tool in proving Theorem \ref{thm:glo}. To recall, this classical result of the theory of semigroups of operators asserts, in its simplest form, that the  convergence of contraction semigroups, such as Feller semigroups in particular,  is equivalent to the convergence of their Laplace transforms, that is, to the convergence of the resolvents of their generators. As applied to the case  of our interest, it says that to establish Theorem \ref{thm:glo} it suffices to show, once we know that $\mathfrak A_\eps$s are Feller generators, that 
\[ \grae \rez{\mathfrak A_\eps}  = \rez{\askew} \]
strongly for all $\lam >0$. We recall furthermore (see e.g. \cite{kallenbergnew} p. 385, see also  \cite{ethier}*{Chapter 4, Theorem 2.5}) that such convergence implies (and is implied by) weak convergence of distributions of the corresponding processes in the Skorokhod space of paths, provided that the initial distributions of these processes converge weakly also. 
 
\begin{remark}\label{rem:1} \rm \ 

\bf (a) \rm A skew Brownian motion can be equivalently characterized by $\gamma \coloneqq  \frac {1-p}p \in (0,\infty)$.  This parameter transforms in a simpler way: we have $\widetilde \gamma = \e^{-2\alpha} \gamma$. Yet different parameter, used preferably by Portenko, is 
\[c \coloneqq  2p -1 = \frac {1-\gamma}{1+\gamma}\in (-1,1),\]
and here we have    
\[ \widetilde c = \tfrac {1- \gamma \e^{-2\alpha}}{1+ \gamma \e^{-2\alpha}} = \tfrac {1- \e^{-2\alpha} + c(1+ \e^{-2\alpha})}{1+ \e^{-2\alpha} + c(1- \e^{-2\alpha})}.\]
 For $c=0$, that is, $\gamma =1$ or $p=\frac 12$, the last formula reduces to that of Portenko, who proved in that case that the limit skew Brownian motion's parameter is $\widetilde c = \tanh \alpha$. Similarly, \eqref{mthm:1} with $p=\frac 12$ reduces to \eqref{intro:2}. Finally, we note that in terms of $\beta\coloneqq -\frac 12 \ln \gamma$ the above transformations take the following elegant form: 
 \begin{align*} \gamma = \e^{-2\beta} \qquad &\longmapsto \qquad \widetilde \gamma = \e^{-2(\alpha+\beta)}, \\ c=\tanh \beta \qquad &\longmapsto \qquad\widetilde c = \tanh (\alpha +\beta). \end{align*}
 
\bf (b) \rm  As already remarked in \cite{manpil}, the assumption of non-negativity of $a_\eps$s in Theorem \ref{thm:glo} is not essential: it suffices to assume that \begin{equation}\label{dod} \sup_{\eps >0} \int_{\R} |a_\eps (x)|\ud x\eqqcolon M <\infty \end{equation} and that \eqref{intro:1} holds with certain constant $\alpha$ which now need not be non-negative. (Of course, in this case $\widetilde p$ o \eqref{mthm:1} can be smaller than $p$.) Hence, in Section \ref{sec:pot} containing the proof of Theorem \ref{thm:glo}, we drop the assumption of positivity and assume \eqref{dod} instead. 

Interestingly (and this will turn out important later), the assumption that $a_\eps$s are continuous at $0$ is also superfluous: we can work with $a_\eps$s that have one-sided limits at this point, provided that condition (c) in the definition of $\dom{\mathfrak A_\eps}$ is changed to read $\frac 12 f''(0+) + a_\eps (0+) f'(0+) = \frac 12 f''(0-) + a_\eps (0-) f'(0-)$. In the proof presented in Section \ref{sec:pot} we use only this weaker assumption on $a_\eps$.    \end{remark}

\section{Proof of Theorem \ref{thm:glo}} \label{sec:pot} 
 \newcommand{\crr}{\int_{-r}^r}
To prove our theorem we  express the resolvents of $\mathfrak A_\eps$ as integral operators
\begin{equation}\label{potm:0} \rez{\mathfrak A_\eps} g (x) = \crr K_{\lam,\eps} (x,y)g(y) \ud y , \qquad \lam,\eps >0, g \in C[-r,r] \end{equation}
where for certain two solutions, say, $k_\eps $ and $\ell_\eps$, the first non-decreasing, the second non-increasing,  of the related Sturm--Liouville eigenvalue problem, we have 
\begin{equation}\label{potm:1} K_{\lam,\eps} (x,y) = \frac{2 k_\eps (x)\ell_\eps (y)}{w_\eps(y)} [x\le y] +  \frac{2k_\eps (y)\ell_\eps (x)}{w_\eps(y)} [x> y], \quad x,y\in [-r,r],\end{equation}  
with $w_\eps \coloneqq  k_\eps ' \ell_\eps - k_\eps \ell_\eps'$. The existence and convergence of these solutions is established in Section \ref{sec:eac}, and the convergence of resolvents --- in Section \ref{sec:cor}. The latter convergence, in conjunction with the Trotter--Kato--Neveu theorem, 
will be used to  prove our theorem.

To recall, it is a classical result (see e.g.  \cite{lions3}*{Section VIII.2.7.1}, \cite{debmik}*{Section 5.10} or \cite{robinson2020}*{Chapter 17})) that the resolvents of Sturm--Liouville operators can be expressed by means of two independent eigenvectors as in \eqref{potm:0} and \eqref{potm:1}.  In the probabilistic context, however, it is natural to use two special eigenvectors: one non-decreasing and one non-increasing (as above) so that in particular the Green's function $K_{\lam,\eps}$ is nonnegative. W. Feller attributes this idea to E. Hille (\cite{fellera3}*{p. 483}, \cite{fellera4}*{p. 13}) end develops it in his fundamental paper \cite{fellera3}*{Section 12} devoted to the classification of boundary conditions for diffusion processes; see also \cite{engel}*{Section VI.4.b}.

\subsection{The existence and convergence of eigenvectors of related Sturm--Liouville problems}\label{sec:eac}
Throughout the rest of this section, instead of working with parameter $p$ that transforms in a more complex way, we work with $\gamma = \frac{1-p}p$. Our first goal, given $\lam >0$, is to find, for each $\eps >0$, a solution to the Sturm--Liouville eigenvalue problem 
\begin{equation}\label{eac:1} \tfrac 12 k_\eps '' + a_\eps k_\eps' = \lam k_\eps \end{equation}
with boundary and transmission conditions
\begin{equation}\label{eac:2} k_\eps '(-r) = 0, k_\eps (-r) = 1 \mquad { and } k_\eps'(0+) = \gamma k_\eps '(0-).\end{equation}
By a solution we understand a function $k_\eps\in C[-r,r]$ such that
 \begin{itemize} 
\item [(a) ] $k_\eps$ is twice continuously differentiable in either of the subintervals $[-r,0]$ and $[0,r]$ separately  (with one-sided derivatives at each intervals' ends), and satisfies equation \eqref{eac:1} there,
\item [(b) ] conditions \eqref{eac:2} are met.  
\end{itemize} 
We note that such a $k_\eps$  need not belong to the domain $\mc D_{p,\eps}$ of $\mathfrak A_\eps$ (with $p= \frac 1{1+\gamma}$) because $k_\eps'(r)$ in general differs from $0$ (in fact, one can argue that $k_\eps'(r)$ never vanishes). Nevertheless, since $k_\eps$ is required to be a continuous function and \eqref{eac:1} is to be satisfied, condition (c) of the definition of $\mc D_{p,\eps}$ holds automatically for a solution $k_\eps$.

\begin{figure}
\includegraphics[scale=0.5]{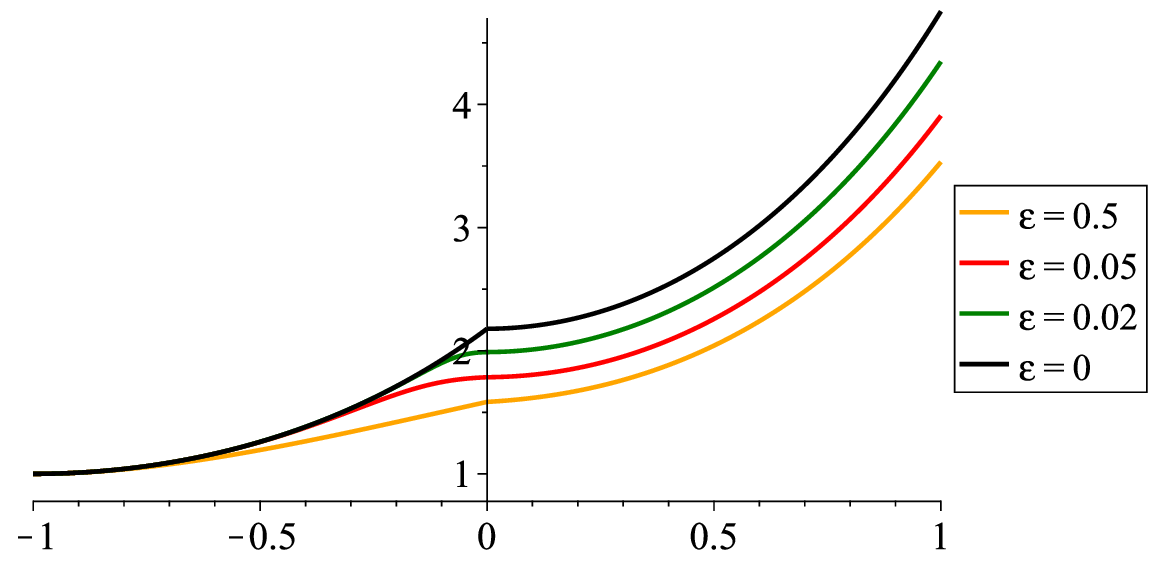}
\caption{\small{Solutions to the Sturm--Liouville eigenvalue problems \eqref{eac:1}--\eqref{eac:2} with $r=\lam =1, \gamma =\frac 14$ and $a_\eps$s obtained via \eqref{zero} from $a$ defined by $a(x)=\e^{\frac 12 x}[x<0]+\e^{-5x}[x\ge 0]$. Since $\widetilde \gamma = \gamma \e^{-2\alpha}$, is significantly smaller than $\gamma$, the limit function (labelled $\eps=0$) has a manifest tip at $0$, reflecting the transformed transmission condition $k'(0+)=\widetilde \gamma k(0-)$.}}\label{sl}
\end{figure}

Here is our main step towards establishing Theorem \ref{thm:glo}.  Notably, our proposition, speaking of the existence and convergence of solutions  $k_\eps$s,  involves a singular perturbation: as $\eps \to 0+$, the influence of the drift cofficient $a_\eps$ concentrates around $x=0$, so that in the limit there is no `regular' drift at all, but the `point' drift at this point is altered, as visible in the transition from the transmission condition $k_\eps '(0+) = \gamma k_\eps (0-)$ to that of $k'(0+) = \widetilde \gamma k(0-)$ (see also Figure \ref{sl}).

\begin{proposition} \label{prop:1} For each $\eps>0$ there is unique solution $k_\eps$ to the problem \eqref{eac:1}--\eqref{eac:2}, and it is strictly increasing. As $\eps\to 0$, the $k_\eps$ converges uniformly to the solution of 
\begin{equation}\label{eac:3} \tfrac 12 k'' = \lam k , \qquad  k'(-r) = 0, k (-r) = 1, k'(0+) = \widetilde \gamma k'(0-),\end{equation}
where $\widetilde \gamma = \e^{-2\alpha} \gamma$.  Together with $k_\eps$s converge their derivatives
but this convergence is merely pointwise and may fail at $x=0\pm$, that is, we have $\grae k_\eps' (x) = k'(x)$ for each $x\in [-r,0) \cup (0,r]$ separately.  
\end{proposition}

For the proof of Proposition \ref{prop:1} we need a couple of lemmas. 

\begin{lemma} A function $k_\eps$ is a solution to the problem  \eqref{eac:1}--\eqref{eac:2} iff 
\begin{equation}\label{eac:4} k_\eps (x) = 1 + 2\lam \int_{-r}^x \int_{-r}^y V_\eps (y,z) k_\eps (z) \ud z \ud y, \qquad x\in [-r,r]\end{equation} where 
\( V_\eps (y,z) \coloneqq  \e^{-2\int_z^y a_\eps (u)\ud u} \left ( \gamma [yz< 0] + [yz\ge  0] \right ).\) \label{lem:1}
\end{lemma}
\begin{proof} 
Multiplying both sides of \eqref{eac:1} by the integrating factor $\e^{2\int_{-r}^x a_{\eps}(z)\ud z}$ and integrating we check that a solution has to satisfy 
\begin{equation}\label{eac:5} k_\eps '(x) = 2\lam \int_{-r}^x \e^{-2\int_{y}^x a_\eps (z)\ud z} k_\eps (y) \ud y, \qquad x \in [-r,0-].\end{equation}
Thus, we must have 
\begin{equation}\label{eac:5a} k_\eps (x) = 1 + 2\lam \int_{-r}^x \int_{-r}^y  \e^{-2\int_{z}^y a_\eps (u)\ud u} k_\eps (z) \ud z \ud y
\end{equation} and this is precisely condition \eqref{eac:4} for $x\in [-r,0]$.

Next, since $k_\eps $ is to satisfy $k_\eps'(0+) = \gamma k_\eps'(0-)$, where $k_\eps '(0-)$ is given by \eqref{eac:5},  we must have similarly 
\begin{align} \nonumber 
k_\eps '(x) &= \gamma k_\eps'(0-) \e^{-2\int_{0}^x a_\eps(z)\ud z} + 2\lam \int_{0}^x \e^{-2\int_{y}^x a_\eps (z)\ud z} k_\eps (y) \ud y\\
&= 2\lam \int_{-r}^x \e^{-2\int_{y}^x a_\eps(z)\ud z} \left (\gamma [y<0] + [y\ge 0]\right ) k_\eps (y) \ud y \nonumber \\
&= 2\lam \int_{-r}^xV_\eps (x,y) k_\eps (y) \ud y, \qquad x\in [0+,r].  \label{eac:6}
\end{align} 
Finally, using the fact that $k_\eps$ is to be continuous at $0$, we see that 
\begin{align}
k_\eps (x) = k_\eps (0)+ 
 2\lam \int_0^x \int_{-r}^y V_\eps (y,z) k_\eps (z) \ud z
\ud y, \qquad x\in [0,r];
\label{eac:7}
 \end{align}
this establishes   \eqref{eac:4}  for such $x$ because $k_\eps (0)$ is given by  \eqref{eac:5a}. 
 

Vice versa, a direct differentiation shows that $k_\eps$ satisfying \eqref{eac:4} is a solution to \eqref{eac:1}--\eqref{eac:2}. 
\end{proof}

In the proof of our next lemma a crucial role is played by the Bielecki-type norm (see \cites{abielecki,jedenipol,edwards,jachymski}) in $C[-r,r]$ defined as 
\[ \|f\|_\omega \coloneqq \max_{-r\le x\le r} \e^{-\omega(x+r)}|f(x)|\]
where $\omega>0$ is a constant. We note that this norm is equivalent to the usual maximum norm, and thus convergence in this norm is just the uniform convergence.

\begin{lemma}\label{lem:2} Let $T_\eps$ be the transformation of $C[-r,r]$ into itself given by 
\begin{equation*} (T_\eps f)(x) = 1 + 2\lam \int_{-r}^x \int_{-r}^y V_\eps (y,z) f(z) \ud z \ud y, \qquad x\in [-r,r]\end{equation*}
where $V_\eps $ has been defined in Lemma \ref{lem:1}. Then, for $M$ of \eqref{dod},
\begin{equation}\label{eac:8}
 \|T_\eps f - T_\eps g \|_\omega \le \frac {2\lam \e^{2M}\max(1,\gamma)}{\omega^2}\| f - g\|_\omega, \qquad f,g 
\in C[-r,r]. \end{equation}

\end{lemma}
\begin{proof} Since $|V_\eps (y,z)|\le \e^{2M} \max (1,\gamma)\eqqcolon \widetilde M$, we see that  
\begin{align*}
 \|T_\eps f - T_\eps g \|_\omega &\le 2\lam \widetilde M \max_{-r\le x\le r} \left  [ \e^{-\omega(x+r)} \int_{-r}^x \int_{-r}^y |f(z) -g(z)|\ud z \ud y\right ] \\ &\le 2\lam  \widetilde M \|f-g\|_\omega \max_{-r\le x\le r} \left [ \e^{-\omega(x+r)} \int_{-r}^x \int_{-r}^y \e^{\omega(z+r)}\ud z \ud y\right ] \\&< \frac {2\lam \e^{2M} \max(1,\gamma)}{\omega^2}\| f - g\|_\omega. \qedhere
\end{align*}
 \end{proof}

Our final lemma speaks of convergence of $T_\eps$. 

\begin{lemma}\label{lem:3} We have $\grae T_\eps f= Tf$, where 
\[ (Tf)(x) = 1 + 2\lam \int_{-r}^x \int_{-r}^y V (y,z) f(z) \ud z \ud y, \qquad x\in [-r,r]\]
and \( V (y,z) \coloneqq  \e^{-2\alpha} \gamma [yz< 0] + [yz\ge   0] .\)
\begin{proof} Since \begin{align*}\|T_\eps f - Tf\|&\le 2\lam \max_{-r\le x\le r} \int_{-r}^x \int_{-r}^y |V_\eps (y,z) - V (y,z)|\, |f(z)| \ud z \ud y\\&=2\lam \int_{-r}^r \int_{-r}^y |V_\eps (y,z) - V (y,z)| \, |f(z)| \ud z \ud y, \end{align*}	 
and the integrand above is bounded by $2\e^{2M} \max(1,\gamma)\|f\|$, it suffices to show that $\grae V_\eps (y,z) = V (y,z)$ for all $y,z\neq 0$. This, however, follows by \eqref{intro:1}. \end{proof}
\end{lemma}

We are now ready to establish Proposition \ref{prop:1}.
\begin{proof}[Proof of Proposition \ref{prop:1}] For $\omega >\sqrt{2\lam \e^{2M} \max (1,\gamma)}$ 
Lemma \ref{lem:2} says that $T_\eps$s are contraction operators (in $C[-r,r]$ equipped with the Bielecki-type norm). It follows that for each $\eps >0$ there is a unique fixed point of $T_\eps$ and so, by Lemma \ref{lem:1}, there is a unique solution $k_\eps$ to the problem \eqref{eac:1}--\eqref{eac:2}. To see that this solution is increasing we first note that it is positive; this is the case  by the following reasons: (a) $k_\eps$ is obtained as a limit of iterates of $T_\eps$ starting from any initial value in $C[-r,r]$, (b) this initial value can thus  be chosen to be nonnegative, and (c) for nonnegative $f\in C[-r,r]$ we have $T_\eps f \ge 1_{[-r,r]}$.  Once positivity of $k_\eps$ is established, a look at \eqref{eac:5} and \eqref{eac:6} reveals that $k_\eps'$ is also positive, because so is $V_\eps$. 


Next, Lemma \ref{lem:2} states also that operators $T_\eps$ have a common contraction constant, and Lemma \ref{lem:3} states that $T_\eps$s converge to $T$. A well-known extension of the Banach Fixed Point Theorem (see e.g. \cites{jedenipol,zmarkusem2b}) says that $T$ is a contraction too (this can also be checked directly), and that the fixed points $k_\eps$ converge to the unique fixed point $k$ of $T$. Now, in view of the form of $V$, Lemma \ref{lem:1} tells us that $k$ is a solution to the problem \eqref{eac:3}. 

Finally, the fact that together with $k_\eps$s converge their derivatives is a direct consequence of the Lebesgue dominated convergence theorem, assumptions  \eqref{intro:1} and \eqref{dod},  formulae \eqref{eac:5} and \eqref{eac:6}, and convergence of $k_\eps$s (note that at $x=0\pm$ the argument fails: we possess no information on convergence of $\int_y^0a_\eps (z) \ud z$ as $\eps \to 0$ for $y<0$).  
\end{proof}

\begin{proposition}\label{prop:2} Given $\lam >0$, for each $\eps >0$, there is a unique solution to the Sturm--Liouville eigenvalue problem 
\begin{equation}\label{eac:10} \tfrac 12 \ell_\eps '' + a_\eps \ell_\eps' = \lam \ell_\eps \end{equation}
with boundary and transmission conditions
\begin{equation}\label{eac:11} \ell_\eps '(r) = 0, \ell_\eps (r) = 1 \mquad { and } \ell_\eps'(0+) = \gamma \ell_\eps '(0-),\end{equation}
and it is strictly decreasing. As $\eps\to 0$, $\ell_\eps$ converges uniformly to the solution of 
\begin{equation}\label{eac:12} \tfrac 12 \ell'' = \lam \ell , \qquad  \ell'(r) = 0, \ell (r) = 1, \ell'(0+) = \widetilde \gamma \ell'(0-),\end{equation}
where $\widetilde \gamma = \e^{-2\alpha} \gamma$.  Together with $\ell_\eps$s converge their derivatives, but this convergence is merely pointwise and may fail at $x=0\pm$. 
 \end{proposition}
\begin{proof}To see this, it suffices to apply Proposition \ref{prop:1} to the case where $a_\eps (x)$ is replaced by $-a_\eps (-x)$ and $\gamma$ is replaced by $\gamma^{-1}$ to obtain a $k_\eps$, and then define $\ell_\eps (x) \coloneqq k_\eps (-x), x\in [-r,r]$. \end{proof}

\subsection{Convergence of resolvents}\label{sec:cor}
In Section \ref{sec:tmt} we have stated that $\mathfrak A_\eps$s defined there are Feller generators. Indeed, it is easy to observe that each of these operators is densely defined and satisfies the maximum principle; a short calculation shows furthermore that the function, say, $f$, defined by the right-hand side of \eqref{potm:0} is a solution to the resolvent equation $\lam f - \mathfrak A_\eps f = g$ for $g\in C[-r,r]$.  This establishes the claim by the Hille--Yosida theorem for Feller semigroups \cites{bass,kniga,jedenipol,ethier,kallenbergnew}. 
In particular, by taking $a_\eps =0$ and modifying $p$ to become $\widetilde p$, we see that $\askew$ is a Feller generator also, and that its resolvent is given by the right hand side of \eqref{potm:0}  with $k_\eps$ replaced by the solution $k$ to \eqref{eac:3} and  $\ell_\eps$ replaced by the solution $\ell$ to \eqref{eac:12}. 

By the Trotter--Kato--Neveu theorem \cites{kniga,knigazcup,ethier,goldstein,pazy,kallenbergnew} to prove our Theorem \ref{thm:glo} it suffices to show that the resolvents of $\mathfrak A_\eps$ converge to the resolvent of $\askew$ (and this establishes weak convergence of the related processes, see \cites{ethier,kallenbergnew}). 

\begin{proposition}\label{prop:3} For each $\lam >0$, 
\[ \grae \rez{\mathfrak A_\eps} = \rez{\askew} \]
in the strong topology. 
\end{proposition}
\begin{proof} Formula \eqref{potm:0} can be rewritten as $\rez{\mathfrak A_\eps} g = 2 \ell_\eps h_\eps^+ + 2 k_\eps h_\eps^- $ where 
\[ h_\eps^+ (x) = \int_{-r}^x \frac{k_\eps(y)}{w_\eps (y)} g(y) \ud y\mquad{ and } h_\eps^- (x) = \int_x^{r} \frac{\ell_\eps(y)}{w_\eps (y)} g(y) \ud y\] for  $x \in [-r,r]$. Hence, in view of Propositions \ref{prop:1} and \ref{prop:2}, our task reduces to showing that $\grae h_\eps^+ =h^+$ and $\grae h_\eps^- = h^-$, where \[ h^+ (x) = \int_{-r}^x \frac{k(y)}{w(y)} g(y) \ud y,\qquad  h^- (x) = \int_x^{r} \frac{\ell (y)}{w (y)} g(y) \ud y, \qquad x \in [-r,r],\] and $w=k' \ell - k \ell '  $.  

Now, a short calculation shows that $w_\eps' = - a_\eps w_\eps$ and thus $w_\eps (x) = C_\eps \e^{-\int_{-r}^x a_\eps (y)\ud y}$, where $C_\eps $ depends in a continuous way on $\ell_\eps (-r)$. Since $\grae \ell_\eps (-r)$ exists,  assumption \eqref{dod} implies that there is a constant $M_1$ such that $\frac 1{w_\eps (x)} \le M_1$ for $x\in [-r,r]$.  Furthermore, since $\grae k_\eps =k$ there is a constant $M_2$ such that $\|k_\eps\| \le M_2$ provided that $\eps $ is sufficiently small. It follows that, for such $\eps$,  the integrand below
\[ \max_{x\in [-r,r]} |h_\eps^+ (x) - h^+ (x)| \le  \int_{-r}^r \left |\frac{k_\eps(y)}{w_\eps (y)}  - \frac{k(y)}{w(y)}\right | |g(y)| \ud y\]
does not exceed a constant, and thus the integral converges to zero by Propositions \ref{prop:1} and \ref{prop:2} and the Lebesgue dominated convergence theorem. Convergence of $h_\eps^-$ to $h^-$ is established similarly.  \end{proof}



\section{A singular perturbation of Walsh's spider process}\label{sec:gene}
\newcommand{\stargraph}[2]{\begin{tikzpicture}
      \node[circle,fill=black,inner sep=0pt,minimum size=2pt] at (360:0mm) (center) {};
    \foreach \n in {1,...,#1}{
        \node [circle,fill=blue,inner sep=0pt,minimum size=3pt] at ({\n*360/#1}:#2cm) (n\n) {};
        \draw (center)--(n\n);} 
\end{tikzpicture}}

\begin{figure}
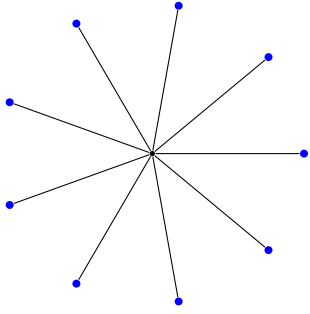

\stargraph{9}{2} 
\caption{Star graph $K_{1,\ka}$ with $\ka=9$ edges, each of them with length $r$.  
}\label{slg}
\end{figure} 

\newcommand{\ceka}{(C[0,r])^\ka}
\newcommand{\ceika}{(C[0,\infty])^\ka}

Walsh's spider process is a generalization of skew Brownian motion,  a process with values on a star graph --- see the works cited above. On each edge, while away from the graph center, this process behaves like a one-dimensional Brownian motion, and upon reaching the center it chooses to continue its motion on the $i$th edge with given probability $p_i$.  In this section we singularly perturb Walsh's spider process by drifts that differ from edge to edge and  show that the limit process is still the Walsh's spider but with transformed probabilities, as visible in \eqref{pta:1}. 

\subsection{Portenko-type approximation of Walsh's spider process}
Let a natural number $\ka \ge 2$ be given, and let $(C[0,r])^\ka$ be the Cartesian product of $\ka$ copies of the space $C[0,r]$ of continuous functions on $[0,r]$. Throughout this section we work with a family $a_\eps, \eps>0$ of members of $\ceka$ such that 
\begin{equation} \label{atek} \max_{i\in \kad} \sup_{\eps >0 } \int_0^r |a_{\eps,i} (u)| \ud u  < \infty \end{equation}
and, for certain reals $\alpha_i$,  
\begin{equation}\label{pta:0} \grae \int_y^z a_{\eps,i} (u) \ud u = \alpha_i [y=0], \quad 0\le y < z\le r, i \in \kad \coloneqq \{1,\dots, \ka\},\end{equation}
$a_{\eps,i}$ denoting the $i$th coordinate of the vector $a_\eps$. Again, such a family can be constructed with the help of $\ka$ continuous, integrable functions $a_i, i \in \kad $ on $[0,\infty)$ by  $a_{\eps,i} (x) = \eps^{-1}a_i (\eps^{-1} x), x \in [0,r]$, and then $\alpha_i = \int_0^\infty a_i(u)\ud u, i \in \kad$.   

Next, let $\mc C\subset \ceka$ be the subspace of $\ceka$ composed of $f\in \ceka$ such that $f_i(0)$ does not depend on $i \in \kad$. This subspace is isometrically isomorphic to the space of continuous functions on the (finite) star graph $K_{1,\ka}$, depicted at Figure \ref{slg}. 

Given a family of functions $a_\eps, \eps >0$ described above and positive parameters $p_1,\dots, p_\ka$ such that $\sum_{i=1}^\ka p_i =1$ we construct operators $\mathfrak A_\eps$ in $\mc C$ as follows. The domain $\dom{\mathfrak A_\eps}$ is composed of elements $f$ of $\ceka$ such that 
\begin{itemize} 
\item [(a) ] each coordinate $f_i$ of $f$ is twice continuously differentiable in $[0,r]$ (with one-sided derivatives at each intervals' ends),
\item [(b) ] $f_i'(r)=0, i \in \kad$ and $\sum_{i=1}^\ka p_if_i'(0)=0$,
\item [(c) ] $\frac 12 f_i''(0) + a_{\eps,i} (0) f_i'(0) $ does not depend on $i\in \kad$.
\end{itemize}
Also, for  $f\in \dom{\mathfrak A_\eps}$ we agree that 
\[ \mathfrak A_\eps f = \seq{\tfrac 12 f_i'' + a_{\eps,i} f_i'}. \]

The main theorem of this section (Theorem \ref{thm:glo2}, below) says, first of all, that each $\mathfrak A_\eps$ is a Feller generator: the related process is a Brownian motion perturbed by the drift $a_{\eps, i}$ in the $i$th edge; while at the inner vertex a particle chooses to continue its motion on the $i$th edge with probability $p_i$. Secondly, the theorem says that as $\eps \to 0$ the semigroups generated by $\mathfrak A_\eps$ converge to the semigroup governing the Walsh's spider process with parameters $\widetilde p_i, i \in \kad $ defined as follows 
\begin{equation}\label{pta:1} \widetilde p_i = \frac {p_i \e^{2\alpha_i}}{\sum_{j\in \mc \kad} p_j \e^{2\alpha_j}}, \qquad i \in \kad .\end{equation}

 \begin{thm}\label{thm:glo2} Operators $\mathfrak A_\eps, \eps >0$ are Feller generators, and 
\[  \grae \e^{t\mathfrak A_\eps} = \e^{t\awalsh}\]
strongly and uniformly for $t$ in compact subintervals of $[0,\infty)$, where $\awalsh$ is the Feller generator in $\mc C$ defined as follows. Its domain $\dom{\awalsh}$ is composed of members $f$ of $\mc C $ such that 
\begin{itemize} 
\item [(i) ] each coordinate $f_i$ of $f$ is twice continuously differentiable in $[0,r]$ (with one-sided derivatives at each intervals' ends),
\item [(ii) ] $f_i'(r)=0, i \in \kad$ and $\sum_{i=1}^\ka \widetilde {p_i}f_i'(0)=0$,
\item [(iii) ] $\frac 12 f_i''(0)$ does not depend on $i\in \kad$.
\end{itemize}
Also, for $f\in \dom{\awalsh}$ we agree that 
\[ \awalsh f = (\tfrac 12 f_i'')_{i\in \kad}. \]
\end{thm}

The Feller process associated with $\awalsh$ is characterized in Remark \ref{rem:2} and Section \ref{rem:andrzeja} further down.

\begin{remark}\rm Formula \eqref{mthm:1} is a particular case of \eqref{pta:1}, corresponding to the case of $\ka =2$. To see this, we note first that viewing $C[-r,r]$ as a subspace of $(C[0,r])^2$ involves a change of perspective. To explain, since the natural isomorphism, say, $\mc I$, defines an image of $(f_1,f_2)\in (C[0,r])^2$ to be the $f\in C[-r,r]$ such that $f(x) = f_1(x)$ for $x\in [0,r]$ and $f(x) = f_2(-x), x\in [-r,0]$,  the second coordinate of $\mathfrak A_\eps (f_1,f_2)$ is $\tfrac 12 f''_2(\cdot ) - a_{\eps,2} (-\cdot) f_2'(-\cdot)$. It follows that the $a_\eps, \eps >0 $ of \eqref{intro:1} are obtained from  $(a_{\eps,1}, a_{\eps,2})$ by the relations $a_\eps(x) = - a_{\eps,2}(-x), x\in [-r,0)$ and $ a_\eps (x) = a_{\eps,1}(x), x \in (0,r]$ (notice the minus sign needed in the first case). Therefore,  $\alpha $ of \eqref{intro:1} and the $\alpha_i,i=1,2$ of \eqref{pta:1} are related by $\alpha = \alpha_1 - \alpha_2$. 
Hence, since $p_1=p$ and $1-p=p_2$, formula \eqref{pta:1} requires that $\widetilde p$  be equal to 
\[ \widetilde p_1= \frac {p\e^{2\alpha_1}}{p\e^{2\alpha_1}+(1-p)\e^{2\alpha_2}} = \frac p{p+(1-p)\e^{-2\alpha}},\] 
as in \eqref{mthm:1}.\end{remark}

\subsection{Proof of Theorem \ref{thm:glo2}.}

We begin our proof of Theorem \ref{thm:glo2} by establishing the following generation result.

\begin{proposition} \label{prop:4} Operators $\mathfrak A_\eps, \eps >0$ are Feller generators in $\mc C$, and so is $\askew$. \end{proposition}
\begin{proof} It is relatively easy to see that $\mathfrak A_\eps, \eps >0$ satisfy the positive-maximum principle and are densely defined. Hence, the problem of proving that $\mathfrak A_\eps$ is a generator comes down to finding, given $\lam,\eps >0$ and $g\in \mc C$, a solution $f$ to \begin{equation}\label{pta:2}\lam f - \mathfrak A_\eps f = g;\end{equation} 
such a solution is necessarily unique. 
In what follows we fix $\lam$ and  $\eps$.   

\subsubsection*{Step 1. Construction of auxiliary functions $k_{\eps,i}$ and $\ell_{\eps,i}$} \

\noindent As in Proposition \ref{prop:1}, we check that, for each $i\in \kad$ there is a unique solution $k_{\eps,i}\in C[0,r]$ of the following Sturm--Liouville problem on 
$[0,r]$:
\[ \tfrac 12 k_{\eps,i}'' + a_{i,\eps} k_{\eps,i}' = \lam k_{\eps,i}, \qquad k_{\eps,i}(0)=0, k_{\eps,i}'(0)=1.\]
Indeed, we can find this solution as a unique fixed point of the transformation $T_{\eps,i} \colon C[0,r]\to C[0,r]$ defined by 
\begin{equation}\label{pta:3} (T_{\eps,i} f)(x) \coloneqq  \int_0^x \e^{-2\int_{0}^y a_{\eps,i} (z)\ud z } \ud y + 2\lam \int_{0}^x \int_0^y \e^{-2\int_{z}^y a_{\eps,i} (u)\ud u} f(z) \ud z\ud y\end{equation} for $x\in [0,r]$ and $f\in C[0,r]$,  $T_{\eps,i}$ being a contraction provided that $C[0,r]$ is equipped with the Bielecki-type norm $\|f\|_{\omega} \coloneqq \max_{0\le x\le r}\e^{-\omega x}|f(x)|$  with sufficiently large $\omega$.  

Next, as in Proposition \ref{prop:2}, we find solutions $\ell_{\eps, i}\in C[0,r]$ to the Sturm--Liouville problem
\[  \tfrac 12 \ell_{\eps,i}'' + a_{\eps,i} \ell_{\eps,i}' = \lam \ell_{\eps,i}, \qquad \ell_{\eps,i}(r)=1, \ell'_{\eps,i}(r)=0.\]

\subsubsection*{Step 2. Construction of auxiliary operators $\rlae^0$} \ 

\noindent With the help of the auxiliary functions constructed in the previous step, we define operators $R_{\lam,\eps}^0\in \mc L(\mc C)$ (the space of bounded linear operators on $\mc C$) as follows
\begin{equation}\label{pta:4} (\rlae^0 g)_i(x) \coloneqq  2k_{\eps,i}(x) \int_x^r \frac{\ell_{\eps,i}(y) g_i(y)}{w_{\eps,i}(y)}\ud y  +   2\ell_{\eps,i}(x) \int_0^x \frac{k_{\eps,i}(y) g_i(y)}{w_{\eps,i}(y)}\ud y    \end{equation}
for $x\in [0,r]$ and $g \in \mc C$, where $w_{\eps,i}\coloneqq k_{\eps,i}' \ell_{\eps,i} - \ell_{\eps,i}'k_{\eps,i}$ (as in the proof of Proposition \ref{prop:3} we need not worry about the denominator being $0$ at any point). A direct calculation persuades us that $f_i \coloneqq (\rlae^0g)_i, i\in \kad$ have the following properties: 
\begin{align*} 
f_i' (x) &=  2k_{\eps,i}'(x) \int_x^r \frac{\ell_{\eps,i}(y) g_i(y)}{w_{\eps,i}(y)}\ud y  +   2\ell_{\eps,i}'(x) \int_0^x \frac{k_{\eps,i}(y) g_i(y)}{w_{\eps,i}(y)}\ud y \\
f_i'' (x) &=  2k_{\eps,i}''(x) \int_x^r \frac{\ell_{\eps,i}(y) g_i(y)}{w_{\eps,i}(y)}\ud y  +   2\ell_{\eps,i}''(x) \int_0^x \frac{k_{\eps,i}(y) g_i(y)}{w_{\eps,i}(y)}\ud y  - 2 g_i(x),   \end{align*}
$x\in [0,r]$. In particular, since $k_{\eps,i}$ and $\ell_{\eps,i}$ are solutions of the appropriate Sturm--Liouville problems, \begin{equation} \label{pta:5} f_i(0)=0,f_i'(r)=0, i\in \kad \mquad{ and } \tfrac 12 f_i'' + a_{\eps,i} f_i' = \lam f_i - g_i.\end{equation} 
We note also use the following relation which is a direct consequence of the formula for $f_i'$: 
\begin{equation}\label{pta:6} (\rlae^0 g)_i'(0) = C_{\eps,i} (g) \coloneqq 2 \int_0^r \frac{\ell_{\eps,i}(y) g_i(y)}{w_{\eps,i}(y)}\ud y. \end{equation}  

\subsection*{Step 3. Construction of $\rlae$; its properties} Let $\mc L_{\lam,\eps} \coloneqq \mathsf 1 - \lam \rlae^0 \mathsf 1$ where $\mathsf 1 = (1_{[0,r]})_{i\in \kad}$ and $1_{[0,r]}(x) =1, x\in [0,r]$; $\mc L_{\lam,\eps}$ is the so-called \emph{exit law} (for the minimal process discussed in Remark \ref{rem:2}, later on). 

We claim that (see \eqref{pta:6} above for the definition of $C_{\eps,i}$) 
\begin{equation} f\coloneqq \rlae g \coloneqq \rlae^0 g + \frac {\sum_{i\in \kad} p_i C_{\eps,i} (g)}{\lam \sum_{i\in \kad} p_i C_{\eps,i} (\mathsf 1)}\mc L_{\lam,\eps} \label{pta:7}\end{equation}
solves \eqref{pta:2} --- notably, $C_{\eps,i}(\mathsf 1)$s are never zero, because $\ell_{\eps,i}$ and $w_{\eps,i}$ are positive. Indeed, first of all, all coordinates of $f$ are twice continuously differentiable since so are the coordinates of $\rlae^0g $ and $\mc L_{\lam,\eps}$. Similarly, $f_j'(r)=0$ because $(\rlae g)_j'(r)=0$ and $(\mc L_{\lam,\eps})'_j (r)=0$. Next, \eqref{pta:6} implies $(\mc L_{\lam,\eps})'_j (0) = - \lam C_j (\mathsf 1)$, and this in turn gives 
\[ f_j '(0) = C_{\eps,j}(g) - C_{\eps,j}(\mathsf 1)\frac {\sum_{j\in \kad} p_j C_{\eps,j} (g)}{\sum_{j\in \kad} p_jC_{\eps,j} (\mathsf 1)}, \qquad j \in \kad.  \]
Multiplying by $p_j$ both sides and summing over $j\in \kad$, yields $\sum_{j\in \kad} p_j f_j'(0)=0$. We have thus shown that $f$ satisfies conditions (a) and (b) of the definition of $\dom{\mathfrak A_\eps}$. 

To complete the proof of the generation theorem for $\mathfrak A_\eps$, let $\mc A_\eps$ be the extension of the operator $\mathfrak A_\eps$ to the domain where condition (c) of the definition of $\dom{\mathfrak A_\eps}$ is satisfied, but the remaining ones need not. Then, relations \eqref{pta:5} tell us that $\rlae^0g$ belongs to the domain of $\mc A_\eps$  for all $g \in \mc C$; indeed, the left-hand side of the differential equation in \eqref{pta:5}, when evaluated at $0$, does not depend on $i$, because neither does its right-hand side. Furthermore, $(\lam - \mc A_\eps)\rlae^0 g = g$. Similarly, $\mc L_{\lam,\eps}$ belongs to this domain and a short calculation reveals that $\mc A_\eps \mc L_{\lam,\eps} = \lam \mc L_{\lam,\eps}$. Hence, $\rlae $ also belongs to the domain of $\mc  A_\eps$ and we have $(\lam - \mc  A_\eps) \rlae g = g$. Now, this relation discloses  that, since $g_i(0)$ and $(\rlae g)_i (0)$ do not depend on $i$, neither does $(\mc A_\eps \rlae g)_i(0)$. This, however, means that $f=\rlae g$ satisfies condition (c) of the definition of $\dom{\mathfrak A_\eps}$ and thus belongs to $\dom{\mathfrak A_\eps}\subset \dom{\mc A_\eps}$ (because, as established above, $f$ satisfies conditions (a) and (b)).  For this reason, the relation turns into $(\lam - \mathfrak  A_\eps) \rlae g = g$, completing the proof of the fact that $\mathfrak A_\eps$ is a Feller generator.

Finally, the statement on $\askew$ is a particular case of that on $\mathfrak A_\eps$ --- to see this, it suffices to take $a_\eps =0$ and replace $p_i$s by $\widetilde p_i$s, which are certainly positive and add up to $1$ as well. \end{proof}

\begin{remark} \label{rem:2} \rm The family $\rlae^0, \lam >0$ defined in Step 2 of the proof is the resolvent of the \emph{minimal} process, that is, of the process which on each edge is a Brownian motion perturbed by the drift $a_{\eps,i}$, but is killed at and thus no longer defined from the first instant it touches the graph center. This information can be used, in conjunction with  the strong Markov property, to guess formula \eqref{pta:7}. To wit, $\rlae, \lam >0$ should be obtained as the limit, say, as $\delta\to 0+$, of the resolvents $\rlae^\delta, \lam >0$ describing processes that continue the minimal process as follows: at the instant when the minimal process reaches $0$, the  process related to $\rlae^\delta, \lam >0$ with probability $p_i$ starts anew at the $i$th edge at the distance $\delta>0$ from the center. The strong Markov property tells us, furthermore (see e.g. the calculation on pp. 2129-2130 in \cite{newfromold}), that 
\[ \rlae^\delta g = \rlae^0 g + \frac{\sum_{i\in \kad} p_i (\rlae^0 g)_i(\delta)}{\lam \sum_{i\in \kad} p_i (\rlae^0 
\mathsf 1)_i(\delta)} \mc L_{\lam,\eps}. \]
Since $ (\rlae^0 g)_i(0)= (\rlae^0 1)_i(0) =0$, to find the limit we  divide both the numerator and the denominator in the fraction above by $\delta$. Then, letting $\delta \to 0$, yields, by 
\eqref{pta:6}, the required formula. 
\end{remark}

The following result is a by-product of the proof of Proposition \ref{prop:4}. 
\begin{corollary}\label{wnos} We have (see  \eqref{pta:7})
\begin{equation} \rez{\mathfrak A_\eps}= \rlae , \qquad \lam,\eps >0.\label{pta:8} 
\end{equation}
In particular, 
\begin{equation}\label{zdzis:1} \rez{\askew} g = \rla^0g +  \frac {\sum_{i\in \kad} \widetilde p_i C_{i} (g)}{\lam \sum_{i\in \kad} \widetilde p_i C_{i} (\mathsf 1)} \mathcal L_\lam, \end{equation}
where 
\begin{itemize}
\item  [1. ] the operators $\rla^0, \lam >0$ are defined by 
\begin{equation}\label{zdzis:2} (\rla^0 g)_i(x) \coloneqq  2k(x) \int_x^r \frac{\ell (y) g_i(y)}{w(y)}\ud y  +   2\ell (x) \int_0^x \frac{k(y) g_i(y)}{w(y)}\ud y,     \end{equation}
for $g\in \mc C,$  $x \in [0,r]$ and $i\in \kad$, 
\item [2. ]$k$ is 
the unique solution of the Sturm--Liouville problem 
\begin{equation}\label{zdzis:3} \tfrac 12 k'' =\lam k , \qquad k(0)=0, k'(0)=1,\end{equation} 
and thus can be calculated explicitly: $k(x) =\frac{\sinh \slam x}{\slam}, x\in [0,r]$; 
\item [3. ]$\ell$ is 
the unique solution of the Sturm--Liouville problem
\begin{equation}\label{zdzis:4}  \tfrac 12 \ell '' =\lam \ell , \qquad \ell (r)=1, \ell' (r)=0, \end{equation} 
and thus can be calculated explicitly: $ \ell (x) = \cosh \slam (r-x), x \in [0,r]$;
\item [4. ]$w \coloneqq  k' \ell  - \ell 'k = (\cosh \slam r)\mathsf 1_{[0,r]}$,  
\item [5. ] $\mc L_\lam \coloneqq \mathsf 1 - \lam \rla^0 \mathsf 1 $, and 
\item [6. ] $ C_{i} (g) \coloneqq (\rla^0 g)_i'(0) = 2 \int_0^r \frac{\ell (y) g_i(y)}{w(y)}\ud y$, that is, 
\begin{equation}
  C_i(g) = \tfrac 2{\cosh \slam r}\int_0^r \cosh \slam (r-y)g_i(y) \ud y  , \qquad i \in \kad, g \in \mc C. \label{ojej1} \end{equation} \end{itemize}

\end{corollary}

We can now present the proof of Theorem \ref{thm:glo2}. 
\begin{proof}[Proof of Theorem \ref{thm:glo2}]
Let us first look at the operators $\rlae^0, \lam >0$ featured in \eqref{pta:7}, and defined in \eqref{pta:4}. By assumption \eqref{pta:0}, for each $i\in \kad$, as $\eps \to 0$, $T_{\eps,i}f$ introduced in \eqref{pta:3} converges (in the norm of $C[0,r])$ to $T_if $ defined by 
\begin{equation}\label{pta:9} (T_{i} f)(x) =  \e^{-2\alpha_{i}} x + 2\lam \int_{0}^x \int_0^y  f(z) \ud z\ud y, \qquad x \in [0,r].\end{equation}
As in the proof of Proposition \ref{prop:1} we deduce, therefore, that $k_{\eps,i}$ converge uniformly to the unique solution of the Sturm--Liouville problem 
\begin{equation}\label{zdzis:5} \tfrac 12 k_{i}'' =\lam k_{i}, \qquad k_{i}(0)=0, k'_{i}(0)=\e^{-2\alpha_i},\end{equation} 
and so do their derivatives, even though the latter functions converge merely pointwise. 
Similarly, $\ell_{\eps,i}$ converge uniformly to the unique solution $\ell$ of \eqref{zdzis:4} 
and so do converge (pointwise) their derivatives. 

Hence, as in the proof of Proposition \ref{prop:3}, we deduce that $\rlae^0$ converge strongly to $\rlaz^0 $ defined as follows (comp. \eqref{pta:4})
\begin{equation}\label{pta:10} (\rlaz^0 g)_i(x) \coloneqq  2k_{i}(x) \int_x^r \frac{\ell (y) g_i(y)}{w_{i}(y)}\ud y  +   2\ell (x) \int_0^x \frac{k_{i}(y) g_i(y)}{w_{i}(y)}\ud y    \end{equation}
for $x\in [0,r]$ and $g\in \mc C$, where $w_{i}\coloneqq k_{i}' \ell - \ell 'k_{i}$. It follows that $\mc L_{\lam,\eps}$ converge to 
\begin{equation}
    \label{eq:Llambda}
    \mc L_{\lam} \coloneqq \mathsf 1 - \lam \rlaz^0 \mathsf 1
\end{equation} 
(in the norm of $\mc C$). Since at the same time $ C_{\eps,i} (g)$ of \eqref{pta:6} converge to $C_{0,i}(g) \coloneqq 2 \int_0^r \frac{\ell(y) g_i(y)}{w_{i}(y)}\ud y$ (we use \eqref{atek} here), we conclude that $\grae \rlae g$  $= \rla  g $ where 
\begin{equation} \rla g = \rlaz^0 g +  \frac {\sum_{i\in \kad} p_i C_{0,i} (g)}{\lam \sum_{i\in \kad} p_i C_{0,i} (\mathsf 1)}\mc L_{\lam} \label{pta:11}, \qquad g \in \mc C. \end{equation} 

We are thus left with showing that this $\rla$ coincides with $\rez{\askew}$. To this end, we note first that the solution $k_{0,i}$ to the problem \eqref{zdzis:5} is related to the solution $k_{i}$ to the problem \eqref{zdzis:3} by  $k_{0,i} = \e^{-2\alpha_i}k_i, i \in \kad$, and this implies  $w_{0,i} = \e^{-2\alpha_i}w_i $. It follows that $\rlaz^0$ of \eqref{pta:10} is nothing else but 
$\rla^0$ of \eqref{zdzis:2}, and in particular $\mc L_\lam $ defined by \eqref{eq:Llambda} is the same as $\mc L_\lam$ defined in Corollary \ref{wnos}. By the same token, $C_{0,i} =  \e^{2\alpha_i}C_i, i \in \kad$ and thus \eqref{pta:11} can be rewritten as 
\[ \rla g = \rla^0 g +  \frac {\sum_{i\in \kad} p_i \e^{2\alpha_i} C_{i} (g)}{\lam \sum_{i\in \kad} p_i \e^{2\alpha_i} C_{i} (\mathsf 1)}\mc L_{\lam} \qquad g \in \mc C. \] 
Since, by the definition of $\widetilde p_i$s,  the right-hand side is clearly the same as the right-hand side of \eqref{zdzis:1}, we are done.
\end{proof}

\subsection{The process generated by $\awalsh$ can be constructed by means of excursions of the reflected Brownian motion on \([0,r]\)} \label{rem:andrzeja}
 Formula \eqref{zdzis:1} provides also a hint on how the process related to $\askew$ can be constructed. Notably, as already remarked in Introduction, It\^o and McKean have originally defined the skew Brownian motion by means of excursions of the reflected Wiener process on the line: each excursion of the former is defined as the excursion of the latter but with a random sign, being positive with probability $p$ and negative with probability $1-p$. Walsh constructs his process in a similar manner (see also \cite{barlow}): he takes an excursion of the reflected Brownian motion and puts it on the $i$th edge of the star graph with probability $p_i$.

 We claim that the process related to $\awalsh$ can be constructed analogously: the only difference is that 
instead of excursions of the reflected Brownian motion on $[0,\infty)$ we use those of the reflected Brownian motion on $[0,r]$, the Brownian motion  with two reflecting barriers.

To begin with, let  $X$ be a Feller process on $K_{1,\ka}$ that behaves like a Brownian motion on each edge and is reflected at the outer nodes of the graph. The strong Markov property implies that the resolvent of $X$ satisfies the relation
\begin{equation}\label{4.3:1}
        R^X_\lambda g(x)= R^0_\lambda g(x)+ \E_x \e^{-\lambda \sigma } R^X_\lambda g(0),\qquad g \in \mathcal C, x \in K_{1,\ka},
\end{equation}
 in which $\sigma$ is the random moment when $X$ reaches the center of the graph for the first time, and   
 $R^0_\lambda, \lam >0$ is the resolvent of the process $X^0$ obtained from $X$ by killing it at $\sigma$. Furthermore, it can be checked that $\rla^0, \lam>0$ coincides with the resolvent $\rla^0,\lam >0$ of \eqref{zdzis:2}. Also, the calculation \(\lam R^0_\lambda \mathsf 1(x)=\E_x\int_0^\sigma \lam \e^{-\lambda t} \ud t= 1-\E_x\e^{-\lambda \sigma},\) shows that $\E_x\e^{-\lambda \sigma}$ equals $1-\lambda R^0_\lambda \mathsf 1 (x)$ and in particular coincides with 
        $\mc L_{\lam} (x)$, introduced in \eqref{eq:Llambda}. More intrinsically, \eqref{4.3:1} now reveals that to completely characterize $\rla^X, \lam >0$ all one needs to know is $\rla^X g(0)$ for $g\in \mathcal C $ and $\lam >0$ and the resolvent $\rla^0, \lam >0$ of \eqref{zdzis:2}.

To continue, let $B^{\textrm{refl}}$ be  the reflected  Brownian motion on $[0,r]$ and let $\rla^{\textrm{refl}}, \lam >0$  be its resolvent. The relation $\rla^{\textrm{refl}}, \lam >0$  satisfies is a particular case of \eqref{4.3:1} with $\ka =1$ and $\rla^0, \lam >0$ describing the minimal process in the interval $[0,r]$. Moreover,  
\begin{equation}\label{zdzis:6}
R^{\textrm{refl}}_\lambda g(0)=\tfrac {C (g)}{\lam  C  (\mathsf 1_{[0,r]})}, \qquad g\in C[0,r], \lam >0,
\end{equation}
 where, as a counterpart of \eqref{ojej1}, 
\(C (g)\coloneqq \frac 2{\cosh \slam r}\int_0^r \cosh \slam (r-y) g(y) \ud y\).

With these preparations out of the way, using $B^{\textrm{refl}}$, we can define the process, say, $W$, with values in the star graph, as follows. To begin with, we demand that, when started away from the graph center, $W$ behaves like $B^{\textrm{refl}}$ on the edge where it started, up to the moment of first reaching the graph center. Also, to describe its further behavior, 
we first  enumerate  excursions of  $B^{\textrm{refl}}$ in a measurable way (to recall, an excursion is a part of the trajectory of $B^{\textrm{refl}}$ that is positive, and starts and ends at $0$, with no zeros in the middle), and then place an excursion (independently of other excursions and of $B^{\textrm{refl}}$) on the $i$th edge of the graph with probability $\widetilde p_i$ of  \eqref{pta:1}. It is thus clear that  $W$ 
behaves like a Brownian motion on each edge and is reflected at the outer nodes of the graph, and so, in particular, its resolvent $\rla^W, \lam >0 $ satisfies \eqref{4.3:1}. Moreover,  
\[ \rla^W g(0) = \sui \widetilde p_i R^{\textrm{refl}}_\lambda g_i(0), \qquad g \in \mathcal C, \lam >0. \]
Finally, since (a) the constant $C_i(\mathsf 1)$, defined in  \eqref{ojej1},  does not in fact depend on $i\in \kad$ and coincides with $C(\mathsf 1_{[0,r]})$ featured in \eqref{zdzis:6},  and (b) the $\widetilde p_i$s add up to $1$, the factor of $\mc L_\lam$ featured in \eqref{zdzis:1} can be rewritten as  
$ \frac {\sum_{i\in \kad} \widetilde p_i C_{i} (g)}{\lam C (\mathsf 1_{[0,r]})}$ and thus equals $\rla^W g(0)$. It follows that $\rla^W, \lam >0$ coincides with $\rez{\askew}, \lam >0$, completing our proof.

\section{A version on the infinite graph}\label{sec:infinite}
\newcommand{\kiki}{K_{1,\ka}^\infty}
\newcommand{\kikr}{K_{1,\ka}^r}
\newcommand{\kikro}{K_{1,\ka}^\rho}
\newcommand{\kikde}{K_{1,\ka}^\delta}
Our final goal is to establish a counterpart of Theorem \ref{thm:glo2} on the infinite star graph 
\[ \kiki \]
with $\ka \ge 2$ edges, that is, on the set obtained from $\ka$ copies of the compactified half-line $[0,\infty]$ by lumping all zeros into one point.  Unfortunately, the methods developed in Sections \ref{sec:pot} and \ref{sec:gene} cannot be translated \emph{verbatim}  to this setting, mainly because there is no handy equivalent of the Bielecki-type norm in the space \[ C[0,\infty] \] of continuous functions on $[0,\infty)$ that have finite limits at infinity.

We will show, however, that the convergence result we are interested in can be proved with the help of Theorem \ref{thm:glo2}, treated as a key lemma, and some simple manipulations on measures. To this end, rather than beginning, as in Section \ref{sec:gene}, with the description of generators and semigroups, we start with the construction of processes that uses those underlying Theorem \ref{thm:glo2} as building blocks. A direct proof of the convergence of the related resolvents is given later, in Section \ref{sec:corbis}. 

\subsection{Construction of approximating processes}\label{sec:ukcja}
Our point of departure is a family $a_{\eps,i}, \eps>0, i \in \kad$ of continuous functions on $[0,\infty)$ such that 
\begin{equation} \label{coa:0} M\coloneqq \max_{i\in \kad} \sup_{\eps >0 } \int_0^\infty |a_{\eps,i} (u)| \ud u  < \infty \end{equation}
and, for certain reals $\alpha_i$,  
\begin{equation}\label{coa:1} \grae \int_y^z a_{\eps,i} (u) \ud u = \alpha_i [y=0], \qquad 0\le y < z, i \in \kad;\end{equation}
recall from \eqref{pta:0} that $\kad \coloneqq \{1,\dots, \ka\}$. As already stated, such a family can be constructed with the help of $\ka$ continuous, integrable functions $a_i, i \in \kad $ on $[0,\infty)$ by  $a_{\eps,i} (x) = \eps^{-1}a_i (\eps^{-1} x), x \ge 0$, and then $\alpha_i = \int_0^\infty a_i(u)\ud u, i \in \kad$.  We are also given positive parameters $p_i,i \in \kad$ with the property that $\sui p_i =1$.

These objects allow constructing the following two types of  processes, of which our main processes of interest are build. 
\begin{itemize} 
\item [1. ] It is clear that \eqref{coa:0} and \eqref{coa:1} imply that, for any $r>0$, the restrictions of $a_{\eps,i}$s to $[0,r]$ satisfy conditions \eqref{atek} and \eqref{pta:0}. Hence, Theorem \ref{thm:glo2} applies, and in particular there are Feller processes generated by the operators $\mathfrak A_\eps$ of this theorem.  These processes, denoted $Y_{\eps,r}$ in what follows, have values in the finite star graph \(\kikr\) with $\ka$ edges of length $r$ (see Figure \ref{slg}) --- the superscript in $\kikr$  is to stress dependence on $r$. 

\item [2. ] Let $C[-\infty,\infty]$ be the space of continuous functions $f\colon \R \to \R$ that have finite limits at $\pm \infty$. Given 
$ i \in \kad$ and  $\eps >0$, we 
let $a_{\eps,i}^{e}$ be the even extension of $a_{\eps, i}$ to the entire real line. With the help of this extended function, we  define the domain of the operator $\mathfrak G_{\eps,i}$  in  $C[-\infty,\infty]$ to be composed of twice continuously differentiable $f$ such that $\frac 12 f'' + a_{\eps,i}^e f'$ belongs to this space, and for such $f$ agree that 
\[ \mathfrak G_{\eps,i} f = \tfrac 12 f'' + a_{\eps,i}^e f'.\]
To take a closer look at this operator, we recall that by assumption the limits $\lim_{x\to \infty}\e^{\pm \mathfrak a_{\eps,i}(x)}$, where $\mathfrak a_{\eps,i}(x) \coloneqq 2 \int_0^x a_{\eps,i} (y) \ud y, x \ge 0$,  exist, are finite and differ from $0$. Hence, 
\[ \int_0^\infty \e^{\pm \mathfrak a_{\eps,i}(x)}\int_0^x \e^{\mp \mathfrak a_{\eps,i}(y)}\ud y \ud x =\infty . \]
In the language of the Feller boundary classification (see the original \cite{fellera4}*{p. 487} or e.g. \cite{ethier}*{p. 366}, \cite{mandl}*{pp. 24--25}) this means that $\infty$ is a natural boundary for the related diffusion process, and by symmetry the same is true for $-\infty$. Theorem 1 on p. 38 in \cite{mandl} asserts thus that $\mathfrak G_{\eps,i}$ is a Feller generator, and Remark 2 on the same page there supplies the information that the related Feller process is honest. We denote this process  $Z_{\eps,i}$; it has values in $\R$.  
\end{itemize}

With processes $Y_{\eps,r}, \eps >0$ (for a fixed $r>0$) and $Z_{\eps,i}, \eps >0, i \in \kad$ under our belt, we construct processes  $X_\eps, \eps >0$ that have values in the infinite graph $\kiki$ via the procedure explained below; as we shall see later, the final result does not depend on the choice of $r$.

\begin{itemize}
\item [($\star$)] If started at one of the points at infinity, that is, one of the outer nodes of $\kiki$, $X_\eps$ stays there for ever. 
\item [($\star\star$)] If started at an $i$th edge at a distance $\ge r$ from the graph center, $X_\eps$ evolves according to the rules of $Z_{\eps,i}$ up to the moment when $Z_{\eps, i}$ reaches $0$ for the first time; from that moment on $X_\eps$ forgets its past and  starts behaving as in rule ($\star\star\star$) below. Of course, in this description we identify points of the half-line $[0,\infty)$ (contained in the  state-space of $Z_{\eps,i}$)  with the corresponding points of the $i$th edge. 
\item [($\star\star\star$)]  If started  at a distance $< r$ from the graph center, $X_\eps$ behaves like $Y_{\eps,r}$ up to the the moment when its distance from the graph center reaches the level $r$ for the first time.  From that moment on $X_\eps$  forgets its past and starts to behave as in point ($\star\star$) above. Here, we see the finite graph $\kikr$ as a subset of $\kiki$:
\[\kikr \subset \kiki. \]
\end{itemize}

A few remarks are here in order. Firstly, constructions of this type are well-known in  stochastic analysis. For example, they are found useful in building 
(a) global solutions to stochastic differential equations with locally Lipschitz coefficients \cite{gandssde}*{Section 6}, and (b) diffusions on manifolds as concatenations of diffusions defined on local charts \cite{stroockvaradhan_book}*{Section 6.6}.

Secondly, the construction works for two reasons: both kinds of processes involved possess the strong Markov property (which allows restarting them at Markov times) and they are locally consistent. The elaborate on the second reason stated succinctly above: a process which happens to be at a distance larger than $0$ but smaller than $r$ from the graph center does not need to know whether it has reached this position and is to proceed by following the rules of $Y_{\eps, r}$ or the rules of $Z_{\eps,i}$ (for some $i$); in this region both rules are the same.

Thirdly, to make sure the so-constructed process is honest we need to rule out the possibility of its explosion.   But explosion could happen in two and only two scenarios. In the first of these, consecutive times of alternating between the graph center and points of distance $r$ from it, and back, could be so small that they form a converging series. In the second, the process might, using the rules of $Z_{\eps,i}$ for some $i$, reach infinity in finite time.

Fortunately, we need not worry of the first scenario, because the time needed to go from the graph center to one of the points that lie at the distance $r$ from the center is a positive random variable, and by the strong Markov property the terms in the series are independent; it follows that the sum of the series is infinite almost surely. Neither do we need to worry of the second scenario; since $\infty$ is a natural boundary for $Z_{\eps,i}$, with probability one it cannot be reached from the regular state-space in finite time.  

Fourthly, and last, we readily check that the construction in fact does not depend on $r$; the final process does not depend on whether we have used $r$ large or small.

\newcommand{\ksy}{{\eps,i}}
\newcommand{\cxn}{\left (x_n\right)_{n\ge 1}}
\subsection{$X_\eps$s as Feller processes}
The following lemma is a key to understanding processes $Z_{\eps,i}$ and $X_\eps$. 
\begin{lemma}\label{tamel:1} For $\eps, r >0$ and $i\in \kad$, let $\sigma_{\eps,i,r}$ be the time needed for $Z_{\eps,i}$ to reach $\pm r$:
\[ \sigma_{\eps,i,r} \coloneqq \inf \{s \ge 0\colon |Z_{\eps,i}(s) | = r\}. \] 
Then, for any $\rho >0$ and $t>0$,
\[ \lim_{r \to \infty} \sup_{\eps>0, i\in \kad} \sup_{|x|\le \rho} \mathsf P_x (\sigma_{\eps,i,r} \le t) =0. \]
 \end{lemma}
\begin{proof} Let, as before, $a_\ksy^e $ be the even extension of $a_\ksy$. The functions $s_\ksy$ and $A_\ksy$ given by 
\begin{align*} s_\ksy (x) &\coloneqq \int_0^x \e^{-\mathfrak a_\ksy (y)} \ud y, \quad  x \in \R, \text { and}\\
 A_\ksy (t) & \coloneqq 2 \int_0^t \e^{\mathfrak a_\ksy \circ s_\ksy \circ Z_\ksy (s)} \ud s, \quad t \ge 0,\end{align*}
where $\mathfrak a_\ksy (x) \coloneqq 2 \int_0^x a_\ksy ^e (y) \ud y, x \in \R$, are the scale function and a strictly increasing continuous additive functional for $Y_\ksy$, respectively. It is well-known that the formula 
\[ B_\ksy (t) \coloneqq s_\ksy \circ Z_\ksy \circ A_\ksy^{-1} (t), \qquad t \ge 0\]
defines then a Brownian motion starting at $s_\ksy (x)$ if $Z_\ksy$ starts at $x$. (The general statement can be found e.g. in \cite{kallenbergnew}*{Thm. 23.9}, one can also combine Theorems 46.12 and 47.1 in \cite{rogers2}. However, in our simple case the result is a reflection of the fact, visible from the following formula for the generator $\mathfrak G_\ksy $ of $Z_\ksy$:  $\mathfrak G_{\eps,i} f (x) = \frac 12 \e^{-\mathfrak a_\ksy(x)} \frac{\mud}{\mud x} (\e^{\mathfrak a_\ksy (x)} \frac {\mud}{\mud x}f(x))$, that by transforming time (this is the role of $A_\ksy$) and space (the role of $s_\ksy$), $Z_\ksy$ can be seen as a standard Brownian motion.)  Hence, $Z_\ksy \overset{d}  = \widetilde Z_\ksy$ (in distribution) where 
\begin{equation}\label{pox:1} \widetilde Z_\ksy (t) \coloneqq s_\ksy^{-1} ( B(A_\ksy (t)) + s_\ksy (x)), \qquad t \ge 0 \end{equation}
and $B$ is a standard Brownian motion starting at $0$.

Moreover, since $-2M\le \mathfrak a_\ksy (x) \le 2 M$ for $x\in \R$ (see \eqref{coa:0}), we have $2\e^{-2M} t \le A_\ksy (t) \le 2 \e^{2M} t, t \ge 0$. By the same token, $\e^{-2M} x \le s_\ksy (x) \le \e^{2M} x $ for $x \ge 0$ and $\e^{2M} x \le s_\ksy (x) \le \e^{-2M} x$ for $x\le 0$, and this renders $\e^{-2M}x \le s_\ksy^{-1} (x) \le \e^{2M} x $ for $x\ge 0$ and $\e^{2M} x \le s_\ksy^{-1} (x) \le \e^{-2M}x $ for $x\le 0$. We conclude, thus, that all the functions $s_\ksy, s^{-1}_\ksy$ and $A_\ksy$ featured in \eqref{pox:1} have a common estimate of growth. 

Therefore, by the formula just mentioned, our thesis reduces to the statement that 
\( \lim_{r \to \infty}  \sup_{|x|\le \rho} \mathsf P(\tau_{x,r} \le t) =0 \),
where $\tau_{x,r}$ is the time needed for $x+ B(t)$ to reach $\pm r$. This, however, is clear: for any $t >0$, 
\( \lim_{r \to \infty} \mathsf P (\max_{0\le s\le t}|x+  B(s)| < r) =1\). \end{proof}

We are now ready to exhibit the Fellerian nature of $X_\eps$. 

\begin{proposition} \label{prop:pox1} Each $X_\eps, \eps >0$ is a Feller process. \end{proposition}

\begin{proof}

Let $\eps>0$ and $t\ge0$ be fixed. We need to show first of all that $g$ given by  $g(x) \coloneqq \mathsf E_x f(X_\eps (t)), x \in \kiki$ belongs to  $C(\kiki)$ whenever $f$ does; $C(\kiki)$ is the space of continuous functions on $\kiki$.  To this end, we fix a nonzero $f \in C(\kiki)$ (for $f=0$ the statement is obvious) and a sequence $\cxn$ of elements of $\kiki$ such that $x_\infty \coloneqq \gra x_n$ exists. 

\emph{Case 1. Finite $x_\infty$.}   

If $x_\infty$ is not one of the outer nodes of $\kiki$, there is a $\rho>0$ such that all $x_n$s are members of $\kikro$. Also, given $\epsilon >0$, one can find an $r >0$ such that, by Lemma \ref{tamel:1}, $\max_{i \in \kad} \mathsf P_y (\sigma_{\eps, i, r} \le t) < \frac {\epsilon}{6\|f\|}$ for any real $y$ with $|y|\le \rho$.  

Now, how could $X_\eps$ started in $\kikro$ differ from $Y_{\eps,r}$? By definition they are the same up to (but not including) the moment when $Y_{\eps,r}$ reaches one of the outer nodes of $\kikr$ for the first time. However, $Y_{\eps,r}$ cannot reach one of these nodes unless $Z_\ksy$, for some $i\in \kad$, goes from the level of $\rho$ to the level of $r$, and the probability of this event has been chosen above to be small. It follows that for $x\in \kikro$, $g(x)$ and $\mathsf E_x f(Y_{\eps,r}(t))1_{\{\sigma_{\eps,i,r}>t\}}$ differ by at most $\frac \epsilon 6$ and by the same token so do  
$\mathsf E_x f(Y_{\eps,r}(t))1_{\{\sigma_{\eps,i,r}>t\}}$ and $\mathsf E_x f(Y_{\eps,r} (t))$. Hence, 
\begin{equation} |g(x) - \mathsf E_x f(Y_{\eps,r}(t))| < \tfrac \epsilon 3, \qquad x \in \kikro, \end{equation}
and, as a consequence, 
\[ |g(x_\infty) - g(x_n)| < \tfrac {2 \epsilon}3 +   |\mathsf E_{x_\infty} f(Y_{\eps,r} (t))-\mathsf E_{x_n} f(Y_{\eps,r}(t))|, \qquad n \ge 1.  \]
Since, by Theorem \ref{thm:glo2}, $Y_{\eps,r}$ is a Feller process in $\kikr$, letting $n\to \infty$ above yields $\limsup_{n\to \infty} |g(x_\infty) - g(x_n)|\le  \frac {2\epsilon}3 < \epsilon$. This shows $\gra g(x_n) = g(x_\infty)$, $\epsilon >0$ being arbitrary.

\emph{Case 2. Infinite $x_\infty$.}  

If $x_\infty$ is one of the outer edges of $\kiki$, say, the $i$th, then without loss of generality we may assume that all $x_n$s lie at the $i$th edge, away from the graph center. Hence, by definition, $X_\eps$ started at these points behaves like $Z_\ksy$ up to, but not including, the moment  
when $Z_\ksy$ reaches $0$ for the first time. 
 Also, arguing as in the proof of Lemma \ref{tamel:1}, that is, using representation  \eqref{pox:1}, we can show that, for any $\rho >0$, 
\[ \lim_{y\to \infty} \mathsf P_y [\mc B_\rho]=1,\] 
where $\mc B_\rho$ is the set where $\inf_{s\in [0,t]} Z_\ksy (s) \ge \rho$. The above relation implies that $X_\eps$ started sufficiently far away from the center is indistinguishable from $Z_\ksy$ with probability close to $1$.

So prepared, given $\epsilon>0$, let us choose $\rho$ so large that \begin{equation} \label{pox:2} |f(x) - f(x_\infty)|< \tfrac \epsilon 3, \qquad x \in V_\rho,\end{equation} 
where $V_\rho$ is the set of points of the $i$th edge of $\kiki$ that lie at a distance at least $\rho$ from the graph center.  Since $\gra x_n = x_\infty, $ the distance, say, $y_n$, of $x_n$ from the graph center tends to infinity. Thus, for sufficiently large $n$ we have, as established above, 
\[  \mathsf P_{y_n} (\mathcal B_\rho ) > 1 - \tfrac \epsilon{3\|f\|}. \] 
As a result, the difference between $g(x_n)$ and  $\mathsf E_{y_n} f^\sharp (Z_\ksy (t)) 1_{\mc B_\rho} $ is less than $\frac \epsilon 3$; $f^\sharp$ featured here is the restriction of $f$ to the $i$th edge of $\kiki$, as  identified with a function on $[0,\infty]$.  Since   $\mathsf E_{y_n} f^\sharp (Z_\ksy (t)) 1_{\mc B_\rho} $, in turn, differs from $f(x_\infty)\mathsf P_{y_n}(\mathcal B_\rho)$, because of  \eqref{pox:2}, by less than $\frac \epsilon 3 \mathsf P_{y_n} (\mathcal B_\rho)\le \frac \epsilon 3$, we see that 
\[ |g(x_n) - f(x_\infty)| \le \tfrac {2\epsilon}3 + |f(x_\infty)| (1 - \mathsf P_{y_n} (\mc B_\rho))< \epsilon.\]
whenever $x_n $ belongs to $V_\rho$, that is, for all sufficiently large $n$. Since $\epsilon$ is arbitrary, this shows that 
$\gra g(x_n) = f(x_\infty) = g(x_\infty)$, as desired. This completes the proof of the fact that $g$ belongs to $C(\kiki)$. 

We are thus left with showing that, for any $f\in C(\kiki)$ and $x\in \kiki$,  $\lim_{t\to 0+} \mathsf E_x f(X_\eps (t)) = f(x)$; we recall that according to the definition of Feller semigroup, we should in fact prove that this convergence is uniform with respect to $x\in \kiki$, but the well-know result says that in this case pointwise convergence implies uniform convergence --- see e.g.  \cite{kallenbergnew}*{p. 369}. However, condition   $\lim_{t\to 0+} \mathsf E_x f(X_\eps (t)) = f(x)$ is clear by the Lebesgue dominated convergence theorem, because $X_\eps$ has continuous paths;  we are done. \end{proof}

\subsection{Generators of $X_\eps$}\label{sec:gox}

It is our next goal to characterize the generators of the processes $X_\eps$. To this end, it will be convenient to think of the space $C(\kiki)$ as a subspace of the Cartesian product $(C[0,\infty])^\ka$ of $\ka$ copies of the space $C[0,\infty]$ introduced at the beginning of this section. In other words, each $f\in C(\kiki)$ will be seen as a vector $\seq{f_i}$ of elements of $C[0,\infty]$ such that $f_i(0)$ does not depend on $i\in \kad$.

This perspective on $C(\kiki)$ allows us to introduce operators $\mathfrak A_\eps$, later to be shown to be generators of $X_\eps$, as follows. The domain $\dom{\mathfrak A_\eps}$ is composed of elements $f$ of $\ceika$ such that 
\begin{itemize} 
\item [(a) ] each coordinate $f_i\colon[0,\infty)$ of $f$ is twice continuously differentiable in $[0,\infty)$ (with one-sided derivatives at $x=0$), the limit $\lim_{x\to \infty} [\frac 12 f_i''(x)+a_\ksy (x) f'_i(x)]$ exists and is finite, 
\item [(b) ] $\sum_{i=1}^\ka p_if_i'(0)=0$,
\item [(c) ] $\frac 12 f_i''(0) + a_{\eps,i} (0) f_i'(0) $ does not depend on $i\in \kad$.
\end{itemize}
Also, for  $f\in \dom{\mathfrak A_\eps}$ we agree that 
\[ \mathfrak A_\eps f = \seq{\tfrac 12 f_i'' + a_{\eps,i} f_i'}. \]

As a preparation for the proof of Proposition \ref{prop:pox2} (presented further down), we recall the following facts established  in \cite{mandl}*{pp. 25--33}. In the work cited they were proved in greater generality; we present them in a way that pertains directly to our scenario.

In the following list $\eps>0$ and $i\in \kad$ are fixed and so is $\lam >0$; the latter, however, is suppressed notationally.  
\begin{itemize} 
\item [1.] There is a unique continuous function $j_\ksy \colon[0,\infty)\to \R$ such that \[ j_\ksy(x)= 1 + 2\lam \int_0^x \e^{-\mathfrak a_\ksy (y)}\int_0^y \e^{\mathfrak a_\ksy (z)} j_\ksy (z) \ud z \ud y , \qquad x \ge 0,\] where, as before, $\mathfrak a_\ksy (x) \coloneqq 2 \int_0^x a_\ksy (y) \ud y, x \ge 0$. Moreover, $j_\ksy (x) \ge \cosh \omega_1 x$ for $x \ge 0$, where $\omega_1 \coloneqq \sqrt {2\lam \e^{-2M}}$. 
\item [2. ] Since  $j_\ksy$ grows   exponentially,  the integral $\int_0^\infty \e^{-\mathfrak a_\ksy (y)} \frac 1{[j_\ksy (y)]^2} \ud y$ is finite, and thus it makes sense to define 
\begin{align*} k_\ksy (x) &\coloneqq j_\ksy (x) \int_0^x \e^{-\mathfrak a_\ksy (y)} \tfrac 1{[j_\ksy (y)]^2} \ud y , \\
 \ell_\ksy (x) &\coloneqq j_\ksy (x) \int_x^\infty \e^{-\mathfrak a_\ksy (y)} \tfrac 1{[j_\ksy (y)]^2} \ud y , \qquad x \ge 0. 
\end{align*}
Then all the three functions $j_\ksy, k_\ksy$ and $\ell_\ksy$ are solutions to the differential equation $\frac 12 f'' + a_\ksy f' = \lam f$. Moreover,  $k_\ksy$ increases whereas $\ell_\ksy$ decreases. The expression $w_\ksy (x) \coloneqq \e^{\mathfrak a_\ksy (x)} (k_\ksy' (x) \ell_\ksy (x) - \ell_\ksy ' (x) k_\ksy (x))$ in fact does not depend on $x$ and equals $w_\ksy (0)=\ell_\ksy (0)\not =0$. 
\item [3. ] For $\omega \in [0,\omega_1)$, $\lim_{x\to \infty} \e^{\omega x}\ell_\ksy (x)=0$. It follows that for any $h\in C[0,\infty]$ one can let
\begin{align} f_\ksy (x) &\coloneqq \tfrac {2\ell_\ksy (x)}{w_\ksy} \int_0^x k_\ksy (y) \e^{\mathfrak a_\ksy (y)} h(y)\ud y  \label{gox:1} \\ &\phantom{=======} + \tfrac {2k_\ksy (x)}{w_\ksy} \int_x^\infty \ell_\ksy (y) \e^{\mathfrak a_\ksy (y)} h(y)\ud y, \qquad x \ge 0. \nonumber \end{align}
It turns out that $f_\ksy$ belongs to $C[0,\infty]$ and $\lam f_\ksy - \frac 12 f_\ksy'' - a_\ksy f_\ksy'= h$.  

\end{itemize}

\begin{proposition}\label{prop:pox2} Each $\mathfrak A_\eps, \eps >0$ is a Feller generator. \end{proposition}

\begin{proof}Checking that $\dom{\mathfrak A_\eps}$ is densely defined is relatively easy. Also, to establish the positive maximum principle, it suffices to consider the following three cases: (a) a positive maximum of an $f\in \dom{\mathfrak A_\eps}$ is attained at the graph center, (b) it is attained at a finite nonzero distance from the center, and (c)~it is attained at one of the infinite nodes. Case (b) is immediate by a classical result of elementary analysis, and (a) can be taken care of as in e.g. \cite{knigazcup}*{p. 17}, leaving us with (c). 

To take care of (c), we claim that the value of $\mathfrak A_\eps f$ can be positive at no infinite node.  For, assume that it is positive at an $i$th node. Then, there is a $\delta_0>0$ such that $\frac 12 f_\ksy''(x) + a_\ksy (x) f'_\ksy (x) = \frac 12 \e^{-\mathfrak a_\ksy (x)}\frac{\mud}{\mud x} (\e^{\mathfrak a_\ksy (x)} f_\ksy'(x))\ge \delta_0$ for $x$ large enough. It follows that there is a $\delta_1$ such that, for such $x$s, $\e^{\mathfrak a_\ksy (x)} f_\ksy'(x) \ge \delta_1 +2\delta_0 \e^{-2M} x$, and so $f_\ksy'(x) \ge \delta_1\e^{-2M} + 2 \delta_0 \e^{-4M}  x$.  As a result,  $\lim_{x\to \infty} f_\ksy'(x) = \infty$. Since this contradicts the fact that $\lim_{x\to \infty} f_\ksy(x)$ exists and is finite, the positive-maximum principle is established. 

We are thus left with showing that, for any $g\in C(\kiki)$ and $\lam >0$, there is an $f \in \dom{\mathfrak A_\eps}$ such that $\lam f - \mathfrak A_\eps f =g$. To this end, we first define $\rlae^0g$ by requiring that \begin{equation}\label{gox:1a} \rlae^0g \coloneqq \seq{f_\ksy}\end{equation} where $f_\ksy$ is given by  
\eqref{gox:1} with $h$ replaced by $g_i$. Next, we note that \[ f_\ksy '(0)= \frac 2{w_\ksy} \int_0^\infty \ell_\ksy (y) \e^{\mathfrak a_\ksy (y)} g_i(y)\ud y \eqqcolon C_\ksy(g),\] and define 
\begin{equation} f \coloneqq \rlae^0 g + \frac {\sum_{j\in \kad} p_j C_{\eps,j} (g)}{\lam \sum_{j\in \kad} p_j C_{\eps,j} (\mathsf 1)}\mc L_{\lam,\eps}, \qquad g \in C(\kiki), \label{gox:2}\end{equation}
where $\mc L_{\lam,\eps} \coloneqq \mathsf 1 - \lam \rlae^0 \mathsf 1$ whereas $\mathsf 1 \coloneqq (1_{[0,\infty)})_{i\in \kad}$ and $1_{[0,\infty)}(x) \coloneqq 1, x\ge 0$.
To show that this $f$ belongs to $\dom{\mathfrak A_\eps}$ and $\lam f - \mathfrak A_\eps f =g$ holds, we argue as in Step 3 of the proof of Proposition \ref{prop:4}: in particular, we check that $p_i f_\ksy'(0)= p_i C_\ksy (g) - p_i \frac {\sum_{j\in \kad} p_j C_{\eps,j} (g)}{ \sum_{j\in \kad} p_j C_{\eps,j} (\mathsf 1)}C_\ksy (\mathsf 1),$ and so summing over all $i\in \kad$ yields $0$. We omit the remaining details.
 \end{proof}

\begin{proposition}\label{prop:pox3} For $\eps >0$, the process $X_\eps$ is generated by $\mathfrak A_\eps$. \end{proposition}

\begin{proof} The proof becomes rather straightforward if the notion of \emph{characteristic operator of Dynkin} is invoked. 
To recall see \cite{rogers}*{p. 256}, given $f\in C(\kiki)$ and $x\in \kiki$ we define
\begin{equation}\label{dynkina} \mathfrak D_{\eps}f(x) \coloneqq \lim_{\delta \to 0} \frac{\mathsf E_x f(X_\eps(\tau_\delta)) - f(x)}{\mathsf E_x \tau_\delta}, \end{equation}
if the limit exists and is finite; here, $\tau_\delta$ is the first time $X_\eps$ exits  the open ball centered at $x$ with radius $\delta$. This definition needs to be modified if $x$ is one of the external nodes of $\kiki$: in this case we simply let $\mathfrak D_\eps f (x)=0$. We say that an $f\in C(\kiki)$ belongs to the domain of the Dynkin operator iff $\mathfrak D_\eps f (x)$ exists at all $x\in \kiki$ and $\mathfrak D_\eps f$ is a member of $C(\kiki)$.

Interestingly, even though more often than not, that is, under quite general mild conditions,  the Dynkin operator extends the generator of a Markov process
(see \cite{dynkinmp}*{ \S 3 of Chapter 5}), for each Feller processes it turns out to coincide with the generator (see \cite{rogers}*{p. 256}, comp. \cite{kallenbergnew}*{Thm. 19.23}).
Thus, our strategy of proving the proposition is to show that $\mathfrak D_\eps$ is an extension of $\mathfrak A_\eps$. Since the former operator is no different than the generator of $X_\eps$ and no generator can be a proper extension of another generator, this will establish our claim.

Let, therefore, $f$ belong to $\dom{\mathfrak A_\eps}$. We will show that $f$ is in the domain of $\mathfrak D_\eps$ and that $\mathfrak D_\eps f = \mathfrak A_\eps f$. 

To this end, let first $x \in \kiki $ lie at an $i$th edge at a positive distance $y$ from the graph center, and let $\delta $ be smaller than $y$. By definition, the process $X_\eps$ started at $x$, up to the moment when it exists the interval $(y-\delta,y+\delta)$ on the $i$th edge, is no different than $Z_\ksy$ of Section \ref{sec:ukcja}. 
Also, $f_i$ can be modified in the interval $[0,y-\delta)$ and extended to the entire left axis in such a way that the new function, say, $f_i^\flat$, belongs to the domain of the generator of $Z_\ksy$. Then the limit spoken of in the definition of the Dynkin operator for $Z_\ksy$ at $f_i^\flat$ and $y$ exists and equals $\frac 12 (f_i^\flat)''(y) + a_\ksy (y) (f_i^\flat)'(y)=\frac 12 f_i''(y) + a_\ksy (y) f_i'(y)$.   Since $f_i$ and $f_i^\flat$ coincide in $(y-\delta,y+\delta)$ and, as already mentioned, up to exiting this interval $X_\eps$ differs not from $Z_\ksy$, we conclude that $\mathfrak D_\eps f(x)$ is well-defined and equals $\frac 12 f_i''(y) + a_\ksy (y) f_i'(y) =\mathfrak A_\eps f(x)$.

The case in which $x$ is the graph center is taken care of similarly. First of all,  we fix $r$, and consider only $\delta <r$. Then, $X_\eps$ starting at the graph center is, up to the time when it reaches one of the outer nodes of $\kikde$,  the same as $Y_{\eps,r}$. Secondly, $f$ can be modified outside of $\kikde$ in such a way that this modification, say, $f^\sharp$, when restricted to $\kikr$ belongs to the domain of the generator of $Y_{\eps,r}$. It follows that the limit involved in the definition of the Dynkin operator for $Y_{\eps,r}$ exists and is equal to the value of the generator of $Y_{\eps,r}$ at $f^\sharp$ and $x$, that is, to the common value of $\frac 12 f_i''(0) + a_\ksy (0) f_i'(0), i \in \kad $. Since the limit just spoken of is manifestly the same as the limit involved in the definition of $\mathfrak D_\eps f(x)$, we conclude that also the latter quantity is the common value of   $\frac 12 f_i''(0) + a_\ksy (0) f_i'(0), i \in \kad $, and this is noting else but $\mathfrak A_\eps (x)$.  

The final case is the simplest: if $x$ is one of the outer nodes of $\kiki$, $\mathfrak D_\eps f(x)$ is by definition $0$.

It remains to check that $\mathfrak D_\eps f$ is a continuous function on $\kiki$, and the only points in doubt in $\kiki$ are the outer nodes. 
We need to check, thus,  that for each $i\in \kad$, the limit $\lim_{x\to \infty} [\frac 12 f_i''(x) + a_\ksy (x) f'_i(x)] $ is zero. 
This, however, is rather simple: the argument presented in the proof of Proposition \ref{prop:pox2} persuades us that the limit  cannot be larger than zero, and its straightforward modification shows that neither can it be strictly negative.  

To summarize: we have established that at each point $x\in \kiki$, the quantity $\mathfrak D_\eps f(x)$ is well-defined and coincides with $\mathfrak A_\eps f(x)$; moreover, $\mathfrak D_\eps f = \mathfrak A_\eps f$ is a continuous function on $\kiki$. This means that $f$ is in the domain of the Dynkin operator for $X_\eps$, and that the value of the Dynkin operator is the same as $\mathfrak A_\eps f$. Since $f\in \dom{\mathfrak A_\eps}$ is arbitrary, this completes the proof. \end{proof}

\subsection{Convergence of $X_\eps$, as $\eps \to 0$}\label{sec:coxe}

Let us take another look at the definition of processes $X_\eps, \eps >0$. Theorem \ref{thm:glo2} says that, as $\eps \to 0$, their building blocks, that is, processes $Y_{\eps,r}$, converge, for each $r>0$, to the Walsh processes on $\kikr$. It is thus natural to expect that $X_\eps$s converge, to a process $X_0$ that can be described as follows.

\begin{itemize}
\item [($\star$)] If started at one of the outer nodes of $\kiki$, $X_0$ stays there for ever. 
\item [($\star\star$)] If started at an $i$th edge at a distance $\ge r$ from the graph center, $X_0$ behaves like a one-dimensional Brownian motion  up to the moment when it reaches the graph center for the first time; from that moment on $X_0$ forgets its past and  starts behaving as in rule ($\star\star\star$) below. 
\item [($\star\star\star$)]  If started  at a distance $< r$ from the graph center, $X_0$ behaves like the Walsh process on $\kikr$ (with parameters $\widetilde p_i , i \in \kad$ specified in \eqref{pta:1}) up to the the moment when its distance from the graph center reaches the level $r$ for the first time.  From that moment on $X_0$ forgets its past and starts to behave as in point ($\star\star$) above. 
\end{itemize}

Arguing as in Section \ref{sec:gox} we check that $X_0$ is a Feller process and its generator, say, $\mathfrak A_0$ can be described thusly:
Its  domain $\dom{\mathfrak A_0}$ is composed of elements $f$ of $C(\kiki)\subset (C[0,\infty])^\ka$ such that 
\begin{itemize} 
\item [(a) ] each coordinate $f_i\colon[0,\infty)$ of $f$ is twice continuously differentiable in $[0,\infty)$ (with one-sided derivatives at $x=0$); the limit $\lim_{x\to \infty} f_i''(x)$ exists and is finite, \item [(b) ]   $\sum_{i=1}^\ka \widetilde p_if_i'(0)=0$,
\item [(c) ] $f_i''(0)  $ does not depend on $i\in \kad$,
\end{itemize} 
 and for such $f$ we agree that $\mathfrak A_0 f = \frac 12 f''$. This allows us to conclude, if it was not clear enough from the very description of $X_0$ given above, that $X_0$ is the Walsh process on $\kiki$ (see e.g. \cite{abap,fromsotows,kinetic,kostrykin2012}). 
 
The characterization of $X_0$ given above in points ($\star$)---($\star\star\star$) is apparently a bit more involved than the classical description of the Walsh process, but it has an obvious advantage over the latter: it exhibits a natural connection with processes $Y_{\eps,r}$. We exploit this advantage in the proof of our third main theorem. 

\begin{thm}\label{thm:glo3} Processes $X_\eps$ converge in distribution to $X_0$, provided that they start at the same point in $ \kiki$. \end{thm}

Before embarking on a proof, we note an important change in perspective involved in the theorem: instead of speaking of convergence of semigroups $\sem{\mathfrak A_\eps}$ we concentrate on the convergence of distributions of processes $X_\eps$. Of course, as we know from the  Trotter--Kato--Neveu--Mackevi\v cius Theorem \cite{kallenbergnew}*{Thm. 19.25}, convergence of semigroups is equivalent to convergence of distributions, but the view on the latter will turn out to be a bit clearer in our case.

We will need the following abstract result, concerning convergence of a family of measures $\mu_{\eps,r}$, defined on a complete separable metric space $E$ and indexed by two parameters: $\eps\ge 0$ and $r\in (r_0,\infty]$, for a certain $r_0\ge 0$.  To connect these measures to the objects under study, fix an $x\in \kiki$, a starting point for all stochastic processes involved, let $r_0$ be the distance of $x$ from the graph center, and think of 
\begin{itemize}
\item $E$ as the space of paths; that is, given $t>0$,  let $E$ be the space of continuous functions  $\chi \colon [0,t]\to \kiki$, such that $\chi (0) =x$, 
\item $\mu_{\eps,r}, \eps >0, r \in (r_0,\infty)$ as the distribution of $Y_{\eps,r}$, $\mu_{\eps,\infty}$ as the distribution of $X_\eps$, $\mu_{0,r}$ as the distribution of the Walsh process on $\kikr$, and  the limit measure $\mu_{0,\infty}$ as the distribution of $X_0$ (all of these processes are seen as random elements having values in $E$).    
\end{itemize}

\begin{lemma}\label{andreya}Suppose that 
\begin{itemize}
\item [(i)] for every $r>0$, the measures $\mu_{\eps,r}$ converge weakly, as $\eps \to 0$, to  $\mu_{0,r}$,
\item [(ii)] for each $r>0$ the is an open set $U_r$ with the following two properties: \begin{itemize} \item [(a) ] as restricted to $U_r$ the measures $\mu_{\eps,r}$ and  $\mu_{\eps,\infty}$ coincide: $
         \mu_{\eps,r}|_{U_{r}} = \mu_{\eps,\infty}|_{U_{r}}, \eps>0$,
        and \item [(b) ] $
          \lim_{ r\to\infty}\sup_{\eps\ge 0}\mu_{\eps,\infty} (E\setminus U_{r})=0$.\end{itemize}
        \end{itemize}
        Then   the weak limit $\grae \mu_{\eps,\infty}$ exists and equals $\mu_{0,\infty}$. 
\end{lemma}

\begin{proof} Since the Prokhorov distance $\ud_P$ is dominated by the total variance distance $\ud_{TV}$, the triangle inequality yields 
\[
      \ud_{P}(\nu_{\eps,\infty}, \nu_{0,\infty})\leq \ud_{TV}(\nu_{\eps,\infty}, \nu_{\eps ,r})+
      \ud_{P}(\nu_{\eps ,r}, \nu_{0 ,r})+\ud_{TV}(\nu_{0,r}, \nu_{0,\infty}).
      \]
Next, sub-point (a) of condition (ii) implies that the total variation distance between $\mu_{\eps,r}$ and $\mu_{\eps,\infty}$ 
does not exceed $\mu_{\eps,\infty}(E\setminus U_r)$ and thus, by sub-point (b), converges to $0$ uniformly in $\eps\ge 0$. Hence, given $\delta >0$, one can choose an $r$ so large that the sum of the first and the third terms in the display above is smaller than $\frac 13\delta$ regardless of what $\eps\ge 0$ is. Condition (i) says, moreover, that the Prokhorov distance between $\mu_{\eps,r}$ and $\mu_{0,r}$ converges to $0$. It follows that for $\eps $ small enough the second term in the display does not exceed $\frac 13 \delta$. This, however, means that $\limsup_{\eps \to 0}       \ud_{P}(\nu_{\eps,\infty}, \nu_{0,\infty})\le \frac 23 \delta <\delta$, and implies $\lim_{\eps \to 0}       \ud_{P}(\nu_{\eps,\infty}, \nu_{0,\infty})=0$, $\delta >0$ being arbitrary. Since the convergence in the Prokhorov distance is equivalent to the weak convergence, we are done.  \end{proof}

\begin{proof}[Proof of Theorem \ref{thm:glo3}] In the notations introduced before Lemma \ref{andreya}, our task is to show that $\grae \mu_{\eps,\infty} = \mu_{0,\infty}$ weakly, and of course to this end we want to use the lemma.  

Condition (i) is satisfied by Theorem \ref{thm:glo2}. To prove (ii) we introduce $U_r \coloneqq \{\chi \in E\colon \max_{s\in [0,t]} |\chi (s)|< r\}$; this is the space of paths that up to time $t$ never touch the outer nodes of $\kikr$. It is clear that, as restricted to $U_r$, the distributions of  $Y_{\eps, r}$ and $X_\eps$ are the same, so that assumption (a) is satisfied. This is just to say that up the first time $\tau_{\eps,r}$ when $X_\eps$ touches one of the outer nodes, this process is indistinguishable from $Y_{\eps,r}$. 
Moreover, $\mu_{\eps,\infty} (E\setminus U_r) = \mathsf P_x (\tau_{\eps,r} \le t)$ and, as a consequence of Lemma \ref{tamel:1}, $\lim_{r\to \infty} \sup_{\eps >0} \mathsf P_x (\tau_{\eps,r}\le t)=0$.  Since, clearly, $\lim_{r\to \infty} \mathsf P_x (\tau_{0,r}\le t)=0$, also condition (b) holds, and we are done.  \end{proof}

 \subsection{Convergence of resolvents}\label{sec:corbis} 
 \newcommand{\comek}{C_\omega [0,n]}

 Emboldened by the simplicity of the proof of Theorem \ref{thm:glo3}, in this section, in keeping with the semigroup-theoretic spirit of the paper,  we give a direct proof of the convergence of the resolvents of $\sem{\mathfrak A_\eps}$ to the resolvent of $\sem{\mathfrak A_0}$. Since the statement on this convergence is, by the Trotter--Kato--Neveu--Mackevi\v cius Theorem, equivalent to our Theorem \ref{thm:glo3}, rather than providing all the details, we restrict ourselves to an extensive sketch  of the full reasoning.     
 

The following lemma provides information on the convergence of the functions defined in points 1.-3. in Section \ref{sec:gox}. 

\begin{lemma}\label{lem:resolvents}  For each $i \in \kad$,  \begin{itemize} \item [(a)]$\grae j_{\eps,i} (x) = \cosh (\slam x)$ uniformly with respect to $x$ in compact subintervals of $[0,\infty)$,  
\item [(b)] $\grae k_\ksy (x) = \frac {\e^{-2\alpha_i}}{\slam} \sinh (\slam x ) $ uniformly with respect to $x$ in compact subintervals of $[0,\infty)$,  

\item [(c)] $\grae \frac {\ell_\ksy (x)}{w_\ksy} =  \e^{-\slam x} $ uniformly with respect to $x\ge 0$,

\item [(d)] $\grae f_\ksy (x) =\frac 1{\slam} \int_0^\infty (\e^{-\slam |x-y|} - \e^{-\slam (x+y)}) h(y)\ud y$, 
uniformly with respect to $x\ge 0$.

\end{itemize}\end{lemma}
\begin{proof} \ \

\bf  (a) \rm Given an $\omega>0$ and a natural $n$, let $\comek$ be the space of continuous functions $f\colon [0,n]\to \R$ equipped with the Bielecki norm \( \|f\|_\omega \coloneqq \max_{x\in [0,n]} \e^{-\omega x}|f(x)|\). The operator $T_{\eps,i}$ given, for $f\in \comek $, by 
\[ (T_{\eps, i} f)(x) \coloneqq  1 + 2\lam \int_0^x \e^{-\mathfrak a_\ksy (y)}\int_0^y \e^{\mathfrak a_\ksy (z)} f (z) \ud z \ud y ,  \qquad x \in [0,n],\] 
maps $\comek$ into itself,  and we check, as in the proof of Lemma \ref{lem:2} that
\[ \|T_\ksy f - T_\ksy g\|_\omega \le \tfrac{2\lam \e^{2M}}{\omega^2} \|f-g\|_\omega, \qquad f,g \in \comek, \omega >0. \] 
It follows that  for $\omega > \omega_0\coloneqq \sqrt{2\lam \e^{2M}}$, $T_\ksy$  is a contraction in $\comek$; $j_\ksy$, as restricted to $[0,n]$ is the unique fixed point of $T_\ksy$.

Also, as in Lemma \ref{lem:3} we check that $\grae T_{\eps,i}f = Tf $ (in the standard supremum norm, and hence in the norm of $\comek$) where 
\[ (Tf)(x) = 1 + 2\lam \int_{0}^x \int_{0}^y f(z) \ud z \ud y, \qquad f\in \comek.\]
Therefore,  $j_\ksy$s, as the fixed points of $T_\ksy$s, converge to the fixed point of $T$, that is, to $\cosh (2\lam \cdot )$, uniformly with respect to $x\in [0,n]$. Since $n$ is arbitrary, this proves (a).

\bf (b) \rm Let 
\[ (S_{\eps, i} f)(x) \coloneqq  (T_\ksy f)(x) + \int_0^x \e^{-\mathfrak a_\ksy (y)}\ud y -1,\qquad  f\in \comek, x \in [0,n],\]
 where $T_\ksy$ was defined in the proof of point (a). It is clear that $S_\ksy$ inherits crucial properties of $T_\ksy$, and in particular is a contraction in $\comek$ provided that $\omega >\omega_0$. 
Since $k_\ksy (0)=0$ and $k_\ksy '(0)=1$, a straightforward calculation establishes that $k_\ksy$, as restricted to $[0,n]$,  is a (necessarily unique) fixed point of $S_\ksy$. Finally, $\grae S_\ksy f = S_i f$ where 
\[ (S_if)(x) = \e^{-2\alpha_i}x + 2\lam \int_{0}^x \int_{0}^y f(z) \ud z \ud y, \qquad f\in \comek, x \in [0,n].\]
Therefore,  $k_\ksy$s, as the fixed points of $S_\ksy$s, converge to the fixed point of $S_i$, that is, to $\frac {\e^{-2\alpha_i}}{\slam} \sinh (2\lam \cdot )$, and we complete the proof as in (a).

\bf (c) \rm Let $c_\ksy \coloneqq \int_0^\infty \e^{-\mathfrak a_\ksy (y)} \frac 1{(j_\ksy (y))^2} \ud y$. Since $j_\ksy (y) \ge \cosh \omega_1 y$, the integral is finite. By the same token, because of (a), $\grae c_\ksy = c_i \coloneqq \e^{-2\alpha_i} \int_0^\infty \frac 1{(\cosh 2\lam y)^2} \ud y= \e^{-2\alpha_i} \left [\frac {\tanh \slam y}\slam \right ]_{y=0}^{y=\infty} = \e^{-2\alpha_i}\frac 1\slam $.  

Now, $\frac{\ell_\ksy (0)}{w_\ksy}=1$ and $\frac{\ell_\ksy'(0)}{w_\ksy} =-\frac 1{c_\ksy}$. As above we can argue therefore that $\frac{\ell_\ksy}{w_\ksy}$ is a unique fixed point of the map $U_\ksy$ defined by 
\[ (U_\ksy f)(x) \coloneqq  (T_\ksy f)(x) - \tfrac 1{c_\ksy} \int_0^x \e^{-\mathfrak a_\ksy (y)}\ud y,\qquad  f\in \comek, x \in [0,n].\]
Since $\grae U_\ksy f = U_if $, where 
\[ (Uf)(x) = 1 - \slam x + 2\lam \int_{0}^x \int_{0}^y f(z) \ud z \ud y, \qquad f\in \comek, x \in [0,n],\]
$\frac{\ell_\ksy}{w_\ksy}$s converge, as $\eps \to 0$, to the fixed point of $U$, that is, to $\exp\{-\slam \cdot\}$.  

This establishes uniform convergence in compact subintervals of $[0,\infty)$, $n$ being arbitrary. However, all $\ell_\ksy$s are non-increasing, and this allows deducing that the limit is in fact uniform on the entire half-line. 

\newcommand{\wenta}{R_{\lam,\eps,i}^\sharp}
\bf (d) \rm For $\lam >0$, let $\wenta$ denote the map $h \mapsto f_\ksy$ (recall, that in $f_\ksy$ of \eqref{gox:1} dependence on $\lam$ is notationally suppressed), and let, as before,  $\mathsf 1_{[0,\infty)}$ be the member of $C[0,\infty]$ defined by $\mathsf 1_{[0,\infty)}(x)=1, x \ge 0$.   It is easy to see that $|\wenta h(x)|\le \|h\| \wenta \mathsf 1_{[0,\infty)} (x), x \ge 0$.   

Next, the relation $\frac{\mud}{\mud x} ( \e^{\mathfrak a_\ksy (x)} \ell_\ksy '(x)) = 2\lam \ell_\ksy (x)  \e^{\mathfrak a_\ksy (x)} $ tells us that the function $x\mapsto \e^{\mathfrak a_\ksy (x)} \ell_\ksy '(x)$, which we know has negative values, is increasing, and thus has a limit as $x\to \infty$. Also, the limit cannot be smaller than zero, as this would contradict $\lim_{x\to \infty} \ell_\ksy (x) =0$ (because $\lim_{x\to \infty} \e^{\mathfrak a_\ksy (x)}=\e^{2\alpha_i}$). Hence, using the relation once again, we check that the second term in $\wenta \mathsf 1 (x)$ equals $- \frac {k_\ksy (x) \ell_\ksy'(x)}{\lam w_\ksy} \e^{\mathfrak a_\ksy (x)}$.  Similarly, the first term amounts to $\ell_\ksy (x) \frac{\e^{\mathfrak a_\ksy (x)} k'_\ksy (x) - 1}{\lam w_\ksy}$, and so recalling the definition of $w_\ksy$ yields
\[ \wenta \mathsf 1_{[0,\infty)} = \tfrac 1\lam \mathsf 1_{[0,\infty)} - \tfrac{\ell_\ksy}{\lam w_\ksy}\le \tfrac 1\lam \mathsf 1_{[0,\infty)}.  \]

It follows that the operators $\wenta, \eps >0, i \in \kad$ ($\lam$ is fixed) are equibounded (with bound $\frac 1\lam$), and so to complete the proof it suffices to show the convergence of $ \wenta h$ to the appropriate limit for $h$ in a dense subset of $C[0,\infty]$.  Moreover, the formula in the display above in conjunction with (c) 
shows that \[ \grae \wenta \mathsf 1 (x) = \tfrac 1\lam -\tfrac 1\lam \e^{-\slam x}=\tfrac 1{\slam} \int_0^\infty (\e^{-\slam |x-y|} - \e^{-\slam (x+y)}) \ud y\]  uniformly in $x\ge 0$. We have thus proved the required convergence in the case of $h=\mathsf 1_{[0,\infty)}$. 

Hence, we will be done once we take care of the limit $\grae \wenta h$ for $h$ with the following property: there is an integer $n=n(g)$ such that $h(x)=0$ whenever $x\ge n$.

For such $h$ and $x\ge n$, the second term in $\wenta h (x)$ vanishes whereas the first reduces to $\frac {2\ell_\ksy (x)}{w_\ksy} \int_0^n k_\ksy (y) \e^{\mathfrak a_\ksy (y)} h(y)\ud y$. Also, by~(b), the integral above converges, as $\eps \to 0$, to $\int_0^n \frac 1{\slam} \sinh \slam y \,  h(y)\ud y$. Thus, by (c), uniformly in $x\ge n$, 
\[ \grae \wenta h (x)= \sqrt{\tfrac {2}{\lam}}\e^{-\slam x}  \int_0^n \sinh \slam y \, h(y)\ud y .\] Here, the right-hand side is nothing else than  \[ \sqrt{\tfrac {2}{\lam}}\e^{-\slam x}  \int_0^x \sinh \slam y \, h(y)\mud y = \tfrac 1{\slam} \int_0^x (\e^{-\slam |x-y|} - \e^{-\slam (x+y)}) h(y)\mud y \]
and the upper limit in the last integral can be replaced by $\infty$, because $h(y)=0$ for $y \ge x \ge n$. 

Similarly for $x\in [0,n]$: in this case, the  first term of $\wenta h(x)$ is a product of two functions, of which the second, $x\mapsto \int_0^x k_\ksy (y) \e^{\mathfrak a_\ksy (y)} h(y)\ud y$, by~(b),  converges, as $\eps \to 0$, to $\frac 1 \slam \int_0^x \sinh \slam y \, h(y)\mud y$ uniformly in $x\in [0,n]$. Thus, by (c), the first term converges to $\sqrt{\tfrac {2}{\lam}}\e^{-\slam x}  \int_0^x \sinh \slam y \, h(y)\mud y $ uniformly in $x\in [0,n]$. Since an analogous analysis establishes the convergence of the second term, we conclude that 
\begin{align*} \grae \wenta h(x)& = \sqrt{\tfrac {2}{\lam}}\e^{-\slam x}  \int_0^x \sinh \slam y \, h(y)\mud y \\ &\phantom{==========}+ \sqrt{\tfrac {2}{\lam}}\sinh \slam x  \int_0^x \e^{- \slam y } h(y)\mud y \\ & = \tfrac 1{\slam} \int_0^\infty (\e^{-\slam |x-y|} - \e^{-\slam (x+y)}) h(y)\ud y, \end{align*}
uniformly in $x\in [0,n]$. This completes our argument. 
\end{proof}
 We are now ready to give an independent, semigroup-theoretic, proof of the following restatement of Theorem \ref{thm:glo3}. 
\begin{thm}\label{thm:glo4} For $g\in C(\kiki)$ and $\lam>0$, let  $\rlae g $ denote the right-hand side of 
\eqref{gox:2}. Then, $\rlae $s converge strongly, as $\eps \to 0$, to the operator $\rla $ defined by
\begin{equation}\label{cor:1} \rla g = \rla^0 g + \frac {\sum_{i\in \kad} p_i\e^{2\alpha_i}  C_{i} (g)}{\lam \sum_{i\in \kad} p_i\e^{2\alpha_i} C_{i} (\mathsf 1)}\mc L_{\lam}, \qquad g \in C(\kiki), \end{equation}
where \begin{itemize} 
\item [(a)] $(\rla^0g)_i (x) \coloneqq \frac 1{\slam} \int_0^\infty (\e^{-\slam |x-y|} - \e^{-\slam (x+y)}) g_i(y)\ud y$, $x \ge 0$, 
\item [(b)] $C_i(g) \coloneqq 2 \int_0^\infty \e^{-\slam y}g_i(y)\ud y$, 
\item [(c)] $\mc L_\lam \coloneqq \mathsf 1 - \rla^0 \mathsf 1$, and, to recall, $\mathsf 1\coloneqq \seq {\mathsf 1_{[0,\infty)}}$.
 \end{itemize}
\begin{proof} Lemma \ref{lem:resolvents} (d) states that the operators $\rlae^0$ of \eqref{gox:1a} converge strongly to $\rla^0$ defined above, and so, in particular, $\grae \mc L_{\lam,\eps} = \mc L_\lam$.  It remains, therefore, to show that $\grae C_\ksy (g) = \e^{2\alpha_i} C_i(g), i \in \kad, g \in C(\kiki)$. This, however, is clear from the definition of $C_\ksy(g)$s and  Lemma \ref{lem:resolvents} (c).      \end{proof}
 \end{thm}
 
Finally, we remark that, as expected, the operators $\rla, \lam >0$ defined in \eqref{cor:1} form the resolvent of the Walsh process with parameters $\widetilde p_i, i \in \kad $ of \eqref{pta:1}. This can be argued directly as in Section \ref{rem:andrzeja}, or by consulting e.g. \cite{abap} Sections 3.2 and 4.1. 

\vspace{0.25cm}
\bf Acknowledgment. \rm The authors would like to thank the Isaac Newton Institute for Mathematical Sciences, Cambridge, for support and hospitality during the programme \emph{Stochastic systems for anomalous diffusion}, where work on this paper was undertaken. This work was supported by EPSRC grant EP/Z000580/1.   

Furthermore, we express our gratitude to El\.zbieta Ratajczyk for beautifying our paper with Figure \ref{sl}, and to M. Portenko for sharing his enormous experience in this field with us.

Finally, A. Pilipenko  thanks  the Swiss National Science
Foundation for partial support  of the paper (grants IZRIZ0\_226875, 200020\_214819, 200020\_200400, and 200020\_192129).

\vspace{-1cm}
\bibliographystyle{plain}
\bibliography{bibliografia1}
\end{document}